\documentclass[reqno]{amsart}

\usepackage[utf8]{inputenc}
\usepackage{hyperref}

\usepackage{changepage}
\usepackage[pdftex]{graphicx}
\usepackage{color}
\usepackage{rotating} 

\usepackage{amsmath}
\usepackage{amsfonts}
\usepackage{amssymb}
\usepackage{amsthm}
\usepackage{mathtools}

\usepackage{framed}
 \usepackage{paralist}
\usepackage[shortlabels]{enumitem}

\usepackage{thmtools}
\usepackage{thm-restate}

\usepackage{hyperref}

\usepackage{cleveref}

\newtheorem{thm}{Theorem}
\newtheorem{lemma}{Lemma}

\theoremstyle{definition}
\newtheorem{Def}{Definition}

\theoremstyle{definition}
\newtheorem{Ex}{Example}
\newtheorem*{Ex*}{Example}

\theoremstyle{definition}
\newtheorem*{Construction}{Construction}

\theoremstyle{remark}
\newtheorem*{remark}{Remark}

\newcommand\compl[1]{V\backslash#1}
\newcommand\Z{\mathbb{Z}}       
\newcommand\Zg{\mathbb{Z}_{>0}}   
\newcommand\Zgeq{\mathbb{Z}_{\geq 0}}           

\newcommand\R{\mathbb{R}}  
\newcommand\Rg{\mathbb{R}_{>0}}  
    
\newcommand\Q{\mathbb{Q}}

\newcommand\mP{\mathcal{P}}

\newcommand\X{\mathcal{X}}

\newcommand\nn{\textnormal}
\newcommand\lin{\textnormal{lin}}
\newcommand\dist{\textnormal{dist}}
\newcommand\vol{\textnormal{vol}}
\newcommand\aff{\textnormal{aff}}
\newcommand\E{\textnormal{E}_{\mathcal{P}}}
\newcommand\G{\mathcal{G}}
\newcommand\vDC{DC-volume }
\newcommand\w{correction volume }
\newcommand\ws{correction volumes }

\newcommand\upper[1]{\textnormal{CDC}(#1)}    %CDC
\newcommand\DC[1]{\textnormal{DC}(#1)}    %CDC

\DeclareMathOperator{\Int}{int}

\DeclareMathOperator{\conv}{conv}
\DeclareMathOperator{\bd}{bd}
\DeclareMathOperator{\C}{Cyl}
\DeclareMathOperator{\DV}{DV}
\DeclareMathOperator{\trl}{trl}

\author[M.H.~Ring and A.~Sch\"urmann]{Maren H.~Ring and Achill Sch\"urmann}

\address{% 
Institute for Mathematics, 
University of Rostock,
18051 Rostock,
Germany}
\email{maren.ring@uni-rostock.de}
\email{achill.schuermann@uni-rostock.de}

\keywords{Ehrhart coefficients, local formula, lattice tiling} 

\subjclass[2010]{52C, 52B, 11H}

\title{Local formulas for Ehrhart coefficients from lattice tiles}
\date{\today}

\begin{document}
\maketitle

\begin{abstract}
As shown by McMullen in 1983, the coefficients of the Ehrhart
polynomial of a lattice polytope can be written as a weighted sum of
facial volumes. 
The weights in such a local formula depend only on the outer normal
cones of faces, but are far from being unique.
In this paper, we develop an infinite class of such local formulas. These are based on choices of fundamental domains in sublattices and obtained by 
polyhedral volume computations. 
We hereby also give a kind of geometric interpretation for the 
Ehrhart coefficients.
Since our construction gives us a great variety of possible local formulas, these can, for instance, be chosen to fit well with a given polyhedral symmetry group.
In contrast to other constructions of local formulas, ours does not rely on triangulations of rational cones into simplicial or even unimodular ones.
\end{abstract}

%%%%%%%%%%%%%%%%%%%%%%%%%%%%%%%%%%%%%%%%%%%%%%%%%%%%%%%%%%%%%%%%%%%%%%%%%%%%%%%%
%%%%%%%%%%%%%%%%%%%%%%%%%%%%%%%%%%%%%%%%%%%%%%%%%%%%%%%%%%%%%%%%%%%%%%%%%%%%%%%%

\section{Overview}\label{sec:Intro}

\subsection*{Notation}
We first fix some notation and recall some basic facts.
Let $V$ be a Euclidean space of dimension~$n$ with inner product  $\langle\cdot,\cdot\rangle$ and let $\Lambda$ be a  lattice in $V$ of rank $n$. For a linear subspace $S\subseteq V$, the set $\Lambda \cap S$ is a lattice in $S$, called the \emph{induced lattice} in~$S$. 
A (strict) \emph{tiling} of a set $A\subseteq V$ is a family  of subsets of $A$  that cover $A$ and have pairwise empty intersections.
A \emph{fundamental domain} for a sublattice $L\subseteq \Lambda$ is a connected and bounded subset $T\subseteq \lin(L)$ of the linear hull of $L$, such that the family of translations $\{x+T: \ x\in L\}$ is a tiling of $\lin(L)$. We further require the intersection of $T$ with any affine subspace to be measurable.
%every $L$-orbit intersects $T$ in exactly one point. For another definition and some examples of fundamental domains see Section \note{}\ref{Construction}.

The \emph{affine hull} $\aff(A)$ of a subset $A\subseteq V$ is the
smallest affine space containing~$A$.
We will mainly work with affine spaces $x+S$ which are translates 
of a linear subspace~$S$ by lattice vectors $x\in\Lambda$.
Here, the sublattice $\Lambda \cap S$ in $S$ is assumed to be of maximal possible rank~$\dim S$.
In these affine spaces the \emph{relative volume} or \emph{lattice
  volume} $\vol(A)$ of a set $A$ is defined as the Lebesgue measure,
normalized in a way that a fundamental domain of $\Lambda \cap S$ has volume~$1$.

For a polyhedron $\mP$, we consider its \emph{face lattice}, that is,
the partially ordered set consisting of all faces of $\mP$ with order given by inclusion, where we consider $\mP$ as a face of itself. We denote the order by $\leq$ and write $f<g$ for faces $f,g$ of $\mP$ if we want to exclude the case $f=g$.  The face lattice is a combinatorial lattice, since for every two faces $f,g$ there exist a unique least upper bound $f\vee g$ called \emph{join} and a unique greatest lower bound $f\wedge g$ called \emph{meet}. $f\vee g$ ist the smallest face that contains both $f$ and $g$, and $f\wedge g$ is given by the intersection $f\cap g$. Here, we formally consider the empty set as a face of $\mP$. Since we never use it, we shorten notation by always implying $f\neq \varnothing$ whenever we talk about faces $f\leq \mP$.

A \emph{rational cone} is a set that is defined by finitely many homogeneous inequalities with rational coefficients with respect to a lattice basis. A cone is called \emph{pointed} if it does not contain any nontrivial linear subspace.
Let $f$ be the face of a polyhedron~$\mP$. The (outer) \emph{normal cone} $N_f$ of $P$ at $f$ is defined as
\begin{displaymath}
N_f:=\{ x\in V : \langle x, y-s\rangle \leq 0 \, \, \forall y\in \mP\}
\end{displaymath}
for any vector $s$ in the \emph{relative interior} of $f$, i.e. the
interior with respect to the affine hull $\aff(f)$. It can be shown that this definition does not depend on the choice of~$s$. If $\mP$ is full dimensional, then all normal cones are pointed cones.
The \emph{polar cone} $C^\vee$ of a cone $C$ is defined as 
\begin{displaymath}
C^\vee:= \{ x\in V : \langle x, y\rangle \leq 0 \, \, \forall y\in C\}.
\end{displaymath}
It contains the linear subspace $C^\perp:=\{  x\in V : \langle x, y\rangle = 0 \, \, \forall y\in C\}$ that we call the \emph{orthogonal space} of $C$. 
Given a face $f$ of a polyhedron $P$, the polar cone $N_f^\vee$ of the
normal cone $N_f$ is in the literature often referred to as the \emph{cone of feasible directions} (cf. \cite{BarvinokBook}, \cite{BerlineVergne}). An example is given in Figure \ref{fig:Basics}. 

\begin{figure}[h] 
\begin{center}
\scalebox{0.5}{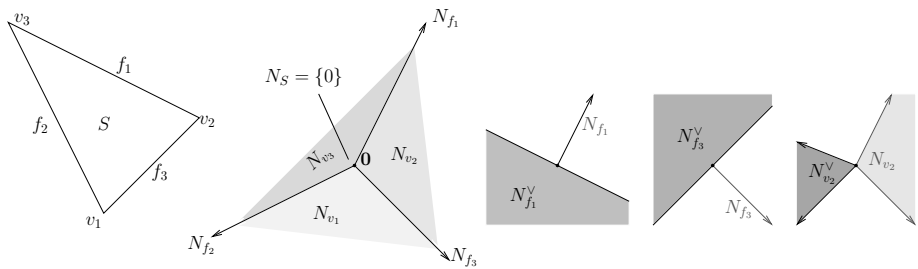}
\end{center}
\caption{Triangle (simplex) $S$ with faces; normal cones of $S$; polar cones of certain normal cones.}
\label{fig:Basics}
\end{figure}

\subsection*{Local Ehrhart formulas}
Let $\mP$ be a \emph{lattice polytope}, that is, $\mP=\conv(v_1, \dots, v_m)$ for some $v_1, \dots, v_m\in \Lambda$. The \emph{Ehrhart polynomial} $\E$ of $\mP$ is the function that maps a nonnegative integer $t\in \Zgeq$ to the number of lattice points in the $t$-th dilate $t\mP$ of $\mP$. It was proven by Ehrhart \cite{Ehrhart} that this function is indeed a polynomial of degree $d:=\dim(\mP)$:
\begin{equation*}
\E(t):=|t\mP \cap \Lambda| = e_dt^d+e_{d-1}t^{d-1}+\dots+e_1t+e_0,
\end{equation*}
with $e_d, \dots, e_0\in \Q$. 
The Ehrhart polynomial plays an important role in areas such as combinatorics, integer linear programming and algebraic geometry. 
It is known that $e_0=1$, that the highest coefficient $e_d$ is the relative volume of $\mP$ and the second highest coefficient $e_{d-1}$ equals half the sum of relative volumes of facets of $\mP$. 
In dimension $d=2$, this gives a full description of the polynomial, generally known as \emph{Pick's formula} \cite{Pick}. The knowledge about the remaining coefficients in higher dimensions is still very limited. 
In 1975, in the context of toric varieties,  Danilov \cite{Danilov} asked whether it is possible to determine the $i$-th coefficient $e_i$ as a weighted sum of relative volumes of the $i$-dimensional faces of $\mP$, where the weights only depend on the normal cone of the faces.  
That such weights  do indeed exist was proved by McMullen in \cite{mcmullen83} (as a variation of his  Theorem 2 with adjustments explained in Section 6 of \cite{mcmullen83} and some of the proofs in \cite{mcmullen78}). Ten years later this was also proved by Morelli \cite{morelli} in the context of toric varietes. 
We note that these existence proofs are not  constructive,  Morelli however also gave a construction where the  weights are rational functions on certain Grassmanians (cf. \cite{PommersheimThomas}). 
The first  construction with rational weights was given by Pommersheim and Thomas in \cite{PommersheimThomas}.
In any case, the weights are far from being unique, which gives rise to the following definition.

\begin{Def} \label{Def:LocalFormula}
A real valued function $\mu$ on rational cones in $V$ is called a \emph{local formula for Ehrhart coefficients} (or \emph{local formula} for short), if for any lattice polytope $\mathcal{P}$ with Ehrhart polynomial $\E (t)=e_d t^d+e_{d-1}t^{d-1}+\dots + e_1t+e_0$, we have
\begin{align*}
e_i=\sum_{\substack{f \leq \mP \\ \dim(f)=i }} \mu(N_f) \vol(f), 
\end{align*}
for all $i\in\{0,\dots, d\}$. %$f\leq \mP$ means that $f$ is lower or equal in the face lattice of $\mP$.
\end{Def}
Since he was the first one to show the existence, local formulas are also  called \emph{McMullen's formulas}.
The normal cones of the faces of a  polytope do not change when taking a dilate by an integer $t\in \Zgeq$. 
The relative volume of a face $f$, however, is homogeneous of degree $\dim(f)$, $\vol(tf)=t^{\dim(f)} \cdot\vol(f) $. For a function $\mu$ on rational cones in $V$, being a local formula for Ehrhart coefficients is thus equivalent to 
\begin{align}
|t\mP \cap \Lambda| = \sum_{f\leq \mP} \mu(N_f)\vol(tf),
\end{align}
for all lattice polytopes $\mP$ and  all $t\in \Zgeq$.
Note that for $t=0$ the relative volumes of the right hand side vanish for all faces except the vertices.
Since both sides of this equation are polynomials, it suffices to show equality for a finite number of values for~$t$. We will use this fact to show that a function is a local formula by showing that equality holds for all sufficiently large $t$.

The word 'local' in the definition is justified, because the function
$\mu$ only depends on the normal cone of the face. That means the only
information taken from the face is its affine hull, respectively the
class of parallel affine spaces it belongs to.
In particular, $\mu$ does neither depend on the size or the shape of a
face, nor on its boundary or other parts of the polytope. 
Advantages of such a local formula are immediate: Properties like positivity of the Ehrhart coefficients can be deduced from the values of $\mu$, without computing the actual coefficients as attempted in \cite{CastilloLiu}. Moreover, the values stay the same for all faces with the same normal cone, so that computations can be done for a whole class of polytopes at once.   For example, the computations of  $\mu$ for the regular permutohedron give the values of all generalized permutohedra as defined in \cite{PostnikovReinerWilliams}. 

Constructions of local formulas have been given by Pommersheim and 
Thomas \cite{PommersheimThomas} and by Berline and Vergne
\cite{BerlineVergne}. While the first is obtained from an expression
for the Todd class of a toric variety, the second depends on the
construction of certain differential operators.
We note that Pommersheim and Thomas also take the normal cone as input for
their local formula, while Berline and Vergne use the transverse cone of a
face, which is an affine version of the cone of feasible directions
modulo its contained linear subspace.

\subsection*{Main results}

In this paper, we develop an infinite class of local formulas for Ehrhart
coefficients. In contrast to the previous constructions it is elementary in the sense that all
information is attained by considering sums and differences of relative volumes.
At this point it is unclear but possible that  other known formulas can be reproduced using our construction. 
In any case, our construction does allow a greater
variety. In Example \ref{Ex:SquareNonStandard} for instance we show that irrational values can occur, in contrast to the local formulas by Pommersheim and Thomas \cite{PommersheimThomas} and by Berline and Vergne \cite{BerlineVergne}.  If desired, it is easily possible though  to restrict our local formulas to rational values as well.
%For proving correctness, we do not need a valuation property for our local formula. Nevertheless we do conjecture that it is a valuation – a proof will follow as soon as possible.
A  computational advantage is that our construction does not rely on simplicial
(cf. \cite{PommersheimThomas})
or even unimodular triangulations 
of cones (cf.~\cite{BerlineVergne}).

We first restrict to the case that the considered rational cones are $pointed$ and the lattice polytopes are full dimensional. For a pointed  rational cone $C$, we first assign a subset of $V$ that we want to call \emph{region of $C$}, denoted by $R(C)$. The construction of this region is quite involved, a thorough description is given in Section \ref{sec:Regions}. 

Using the region $R(C)$, we want to determine the value $\mu(C)$. To give an intuition how this can be achieved, 
we interpret the number of lattice points in $t\mP$ as the volume of all translates of a fundamental domain $T$ of $\Lambda$ by the lattice points in $\mP$:
\begin{equation} \label{Eqn:DCIntro}
|\Lambda \cap t\mP| = \sum_{x\in \Lambda\cap t\mP} \vol(x+T) = \vol\underbrace{((\Lambda\cap t\mP)+T)}_{=:\DC {t\mP}}.
\end{equation}
The first equation holds, since by definition $\vol(T)=1$ for any fundamental domain~$T$ of~$\Lambda$, and the second equation follows from $(x+T)\cap (y+T) = \varnothing$ for all $x,y\in \Lambda$ with $x\neq y$. 
We call the set $\DC {Q} :=(\Lambda \cap Q)+T$ 
a \emph{(fundamental) domain complex} of the polyhedron $Q$. % (cf. Figure \ref{fig:Tiling&DC}, left).

\begin{figure}[h] 
\begin{center}
\scalebox{0.52 }{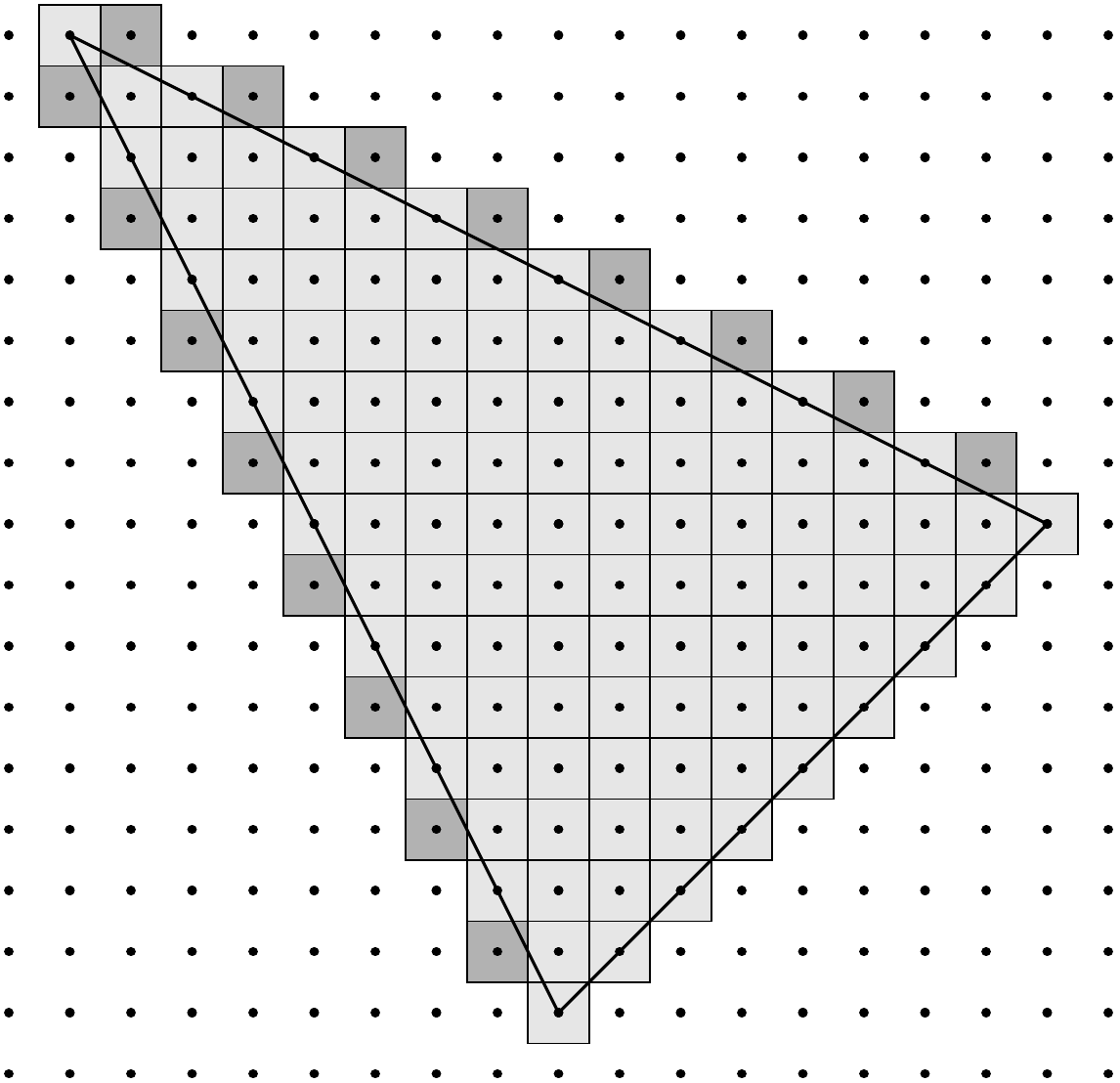} \ % \qquad% \qquad 
\scalebox{0.52 }{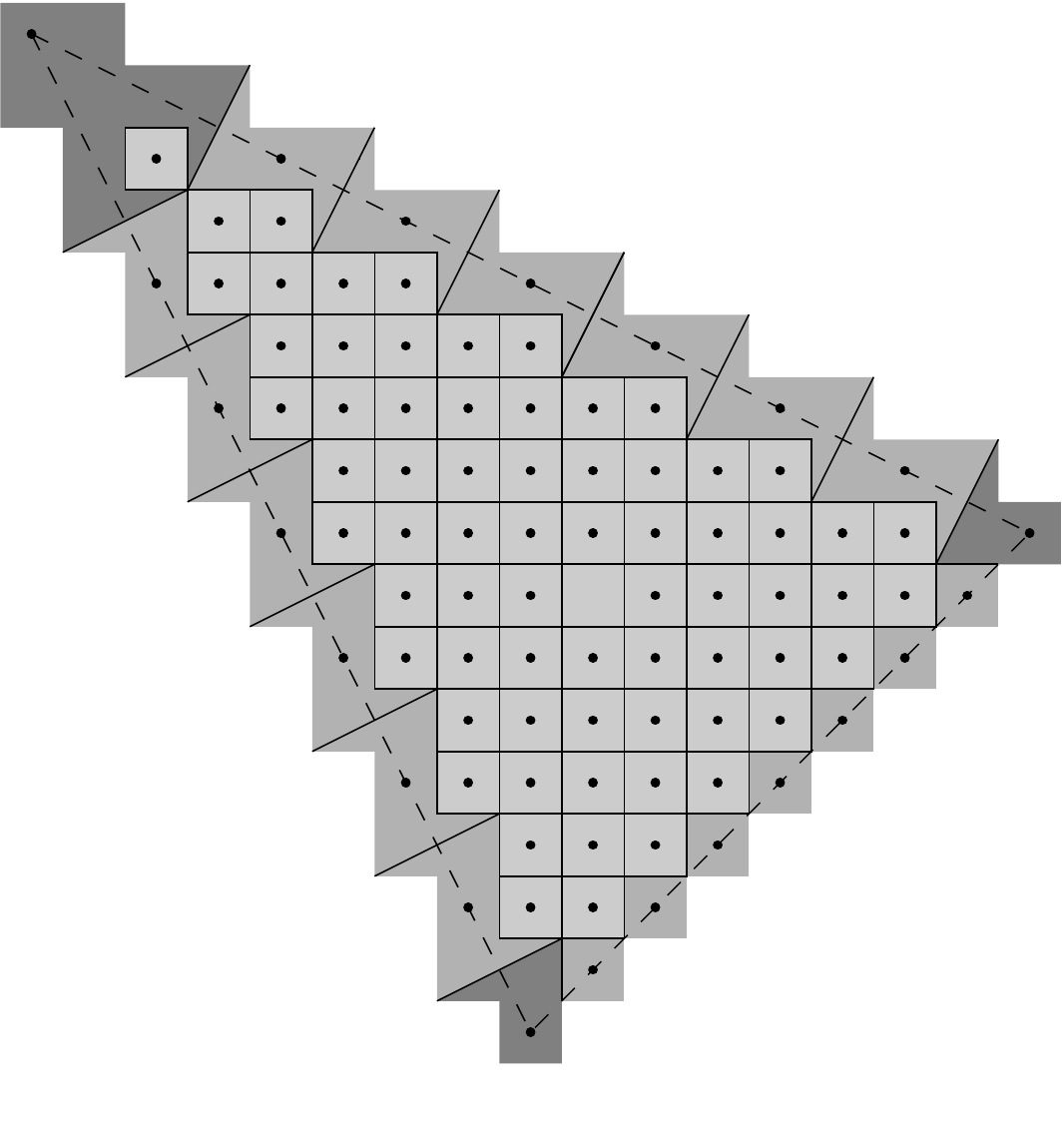}
\end{center}
\caption{ \textbf{Left:}  The covering domain complex of the dilated triangle $4S$  (light and dark grey) and its domain complex (light grey); \textbf{right:}  Tiling of $\upper {4S}$} by regions corresponding to the normal cones of $S$ 
\label{fig:Tiling&DC}
\end{figure}

 To determine the values of $\mu$ properly, we will need correcting terms computed outside of the domain complex. This leads to the following definition of the \emph{covering domain complex} of a polyhedron $Q$, short  $\upper Q $ as
\begin{equation*}
\upper Q := \bigcup_{\substack{x\in \Lambda \\ (x+T)\cap Q\neq \varnothing}} (x+T).
\end{equation*}
It covers $Q$ and can be seen as an upper bound of the domain complex, see Figure~\ref{fig:Tiling&DC}, left.

The most important property of the regions will be the following: Let  $\mP$ be a full dimensional
 lattice polytope %{\red $T$ a fundamental domain of the lattice $\Lambda$}
 and $f\leq \mP$ a face. For $t\in \Zg$ we define the set $\X(tf)$ of all \emph{feasible lattice points in} $t f$  as the finite set of lattice points in the dilated face $tf$ that are 'far enough'  from its boundary (a precise definition can be found in Section \ref{sec:Regions}). Then we have a tiling  of the covering domain complex into regions:

\begin{restatable}[Tiling]{thm}{Tiling}
 \label{Thm:TilingIntro}
Let $\mP\subseteq V$ be a full dimensional 
lattice polytope. There exists a $t_0\in \mathbb{Z}_{>0}$ such that for each $t\geq t_0$   we have a tiling of  $\upper {t\mP}$
into translated regions of the form
\begin{equation*}
\{x+R(N_f) : \; f\leq \mP, \, x\in \X(tf)\}.
\end{equation*}
%\begin{align}
%V=\bigcup_{f \leq \mP} \; \bigcup_{x\in \X(tf)} \big( x+R(N_f) \big)
%\end{align}
\end{restatable}
See Figure \ref{fig:Tiling&DC} (right) for an example in dimension 2 (cf. Section \ref{sec:compsym}, Example \ref{Ex:TriangleStandard}).

%The set $\X(tf)$ of all feasible lattice points in $tf$ might be empty for small $t$, such that the condition $t\geq t_0$ is  necessary. 

By taking the volume of the respective
part of the domain complex in each region of the tiling in
Theorem~\ref{Thm:TilingIntro} (cf. Figure \ref{fig:Tiling&DC}, right), we get
\begin{equation} \label{Eqn:DCTilingIntro}
|\Lambda \cap t\mP|= \vol(\DC {t\mP}) = \sum_{f\leq \mP} \sum_{x\in \X(tf)}\underbrace{ \vol \big( (x+R(N_f)) \cap \DC {t\mP} \big)}_{(*)}.
\end{equation}

It turns out (cf. Section \ref{sec:ProofLocForm}) that ($*$) can be defined only in terms of the cone $N_f$ 
which leads to the definition of $v_C$, the \emph{\vDC} in $R(C)$, for pointed rational cones $C$:
\begin{equation} \label{Eqn:DefDCVolume}
v_{C}:= \vol\big( R(C) \cap \DC {C^\vee} \big).
\end{equation}
For an illustration, see Figure \ref{fig:vN}.

\begin{figure}[h] 
\begin{center}
\scalebox{0.5}{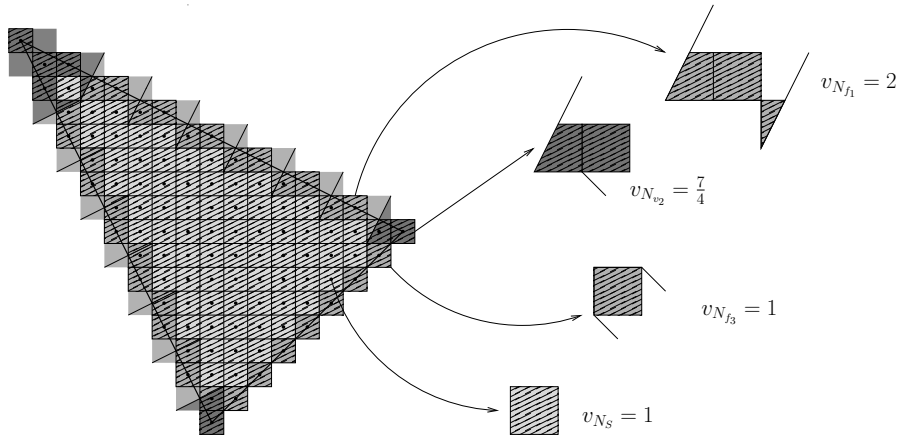}
\end{center}
\caption{Values for the \vDC $v_C$ for normal cones $C=N_f$ of certain faces $f$ of the triangle $S$.}
\label{fig:vN}
\end{figure}

This already looks a lot like a local formula, especially since
$|\X(tf)|$ behaves like $\vol(t f)$ in the limit $t\rightarrow
\infty$. In fact, $|\X(tf)|$ equals $|\Lambda\cap tf|$ 
minus some lower order terms.
To achieve exactness, we use $v_C=v_{N_f}$ together with a
\emph{\w} defined by 
\begin{equation} \label{Eqn:DefCorrectionVolume}
w^C_K:=\vol (R(C) \cap (K^\perp \cap C^\vee))
\end{equation}
for faces $K<C$. Exemplary values with illustrations are given in
Figure \ref{fig:w}. Note that unlike $v_C$, the correction term
$w^C_K$ measures a volume in $K^\perp$, which is only full-dimensional
if $K$ is the \emph{trivial cone} $C_0:=\{0\}$. Note that here 
$N_{\mP}=\{0\}$, since for the moment we are still assuming $\mP$ to be full dimensional.

\begin{figure}[h] 
\begin{center}
\scalebox{0.6 }{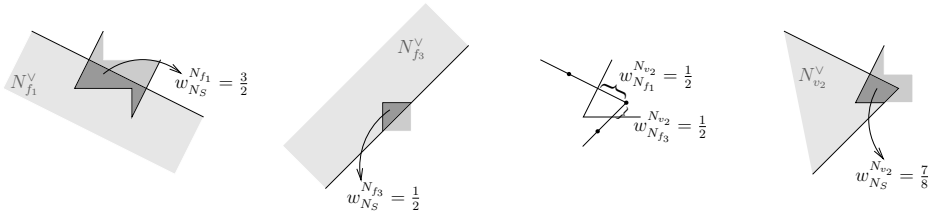} 
\end{center}
\caption{Values for the \w $w^C_K$ with $C=N_f$ and $K=N_g$ for certain faces $f<g$ of $S$. }
\label{fig:w}
\end{figure}

Equation (\ref{Eqn:DCTilingIntro}) thus yields
\begin{equation}
|\Lambda\cap t\mP|=  \sum_{f\leq \mP} \sum_{x\in \X(tf)} v_{N_f} =  \sum_{f\leq \mP}  v_{N_f} \cdot |\X(tf)|.
\end{equation}

Using these notations, the function $\mu$  can be defined on pointed rational cones in~$V$.
We define it by induction on the dimension of the cone, starting with the trivial cone $C_0=~\{0\}$ by setting
\begin{equation} \label{Eqn:muC0Intro}
\mu(C_0):=1.
\end{equation}
For a pointed rational cone $C\subseteq V$ with $\dim(C)\!\geq\! 1$, we  define
\begin{equation} \label{Eqn:muIntro}
\mu(C):=v_C - \sum_{K< C} w^C_K \cdot \mu(K).
\end{equation}

For a rational cone $C\subseteq V$ that is not pointed, but contains a maximal nontrivial linear subspace~$U$, we can consider the pointed cone $C':=C\cap U^\perp$  in $U^\perp$, where we consider $U^\perp$ as a Euclidean space equipped with the induced inner product and the  lattice $\Lambda\cap U^\perp$. We can then construct $R(C') \subseteq U^\perp$ and set
 \begin{equation*}
\mu(C):=\mu(C').
\end{equation*} 

That leads us to the main result of this work:

\begin{thm}[Local Formula] \label{Thm:MuIntro}
The function $\mu$ on 
rational cones in $V$ as defined in Equations (\ref{Eqn:muC0Intro}) and (\ref{Eqn:muIntro}) is a local formula for Ehrhart coefficients. 
\end{thm}

The given construction for local formulas has several nice
properties. It is more basic than previous constructions in the sense
that it is based on basic notions from polyhedral geometry. In a way,
it hereby also gives a geometric meaning to the coefficients of the Ehrhart polynomial.

Another nice property of the construction is the freedom of choice of
a fundamental domain in each occurring sublattice. An interesting
observation is, for example, that this construction also allows
irrational values, which can be achieved simply by taking a
fundamental domain and shifting it by an irrational vector. This shows
that the range of this construction is wider than the one of others previously
known. It is yet to be determined how extensive this variety actually
is, whether, for example, hitherto existing constructions can be
described in terms of our construction using certain fundamental
domains. In~\cite{CastilloLiu} it is shown that all local formulas
that are invariant under the standard action of the symmetric group
must agree on certain polytopes that are invariant with respect to
this action themselves. As described in Section \ref{sec:Regions}, a
natural choice for fundamental domains are the so-called
\emph{Dirichlet--Voronoi cells} that depend on a chosen inner
product. Given the lattice $\Z^n\subseteq \R^n$, and taking the
Dirichlet--Voronoi cell given by the standard inner product, one gets
a fundamental domain and thus a local formula that is symmetric about
the origin.  This principle can be extended to other symmetries and,
as we will discuss in Section \ref{sec:compsym}, it leads to new possibilities to exploit symmetries in given polytopes.  

Other local formulas (\cite{mcmullen83}, \cite{morelli}, \cite{PommersheimThomas}, \cite{BerlineVergne}) are known to have the nice property of being a valuation, that is, they satisfy
\begin{align*}
\mu(C \cup K) + \mu(C\cap K) = \mu(C) + \mu(K)
\end{align*}
for convex rational cones such that $C\cup K$ is also a convex rational cone. 
We \textbf{conjecture} that our function $\mu$ is a valuation as well.
However, a proof of this property seems non-trivial at this point. 
In contrast to other local formulas we do not need the valuation property for our proofs or the computation of  $\mu$.

As for other known constructions of local formulas, ours can be extended to rational polytopes $P$:
Given the normal cone $N_f$ and the translation class $\trl(f)$ for each face $f$ of $P$, the construction can be adjusted such that 
\begin{equation*}
|P\cap \Lambda|=\sum_{f\leq P}\mu(N_f, \trl(f)) \vol(f),
\end{equation*}
which yields a formula for the Ehrhart quasi-polynomials of a rational polytope. A more involved proof 
of this property is in progress and will be presented in \cite{Ring}.

Our paper is organized as follows. In Section~\ref{sec:Regions} we
give a construction of regions for pointed rational cones based on a
choice of fundamental domains. We introduce an important class of
examples coming from Dirichlet--Voronoi cells of lattices. 
In Section~\ref{sec:compsym} we then give several descriptive, two-dimensional 
examples with concrete values for~$\mu$. Moreover, we 
describe a canonical way to exploit symmetries using certain
Dirichlet--Voronoi cells. In Section~\ref{sec:Tiling}, we give a proof
for Theorem~\ref{Thm:TilingIntro}, showing that a tiling of~$\upper \mP$ can be
gained from translates of regions corresponding to normal cones of a
full dimensional lattice polytope $\mP$. 
In Section~\ref{sec:ProofLocForm} we close with a proof of Theorem~\ref{Thm:MuIntro},
showing that the constructed functions~$\mu$ on rational cones are indeed local formulas.

%%%%%%%%%%%%%%%%%%%%%%%%%%%%%%%%%%%%%%%%%%%%%%%%%%%%%%%%%%%%%%%%%%%%%%%%%%%%%%%
%%%%%%%%%%%%%%%%%%%%%%%%%%%%%%%%%%%%%%%%%%%%%%%%%%%%%%%%%%%%%%%%%%%%%%%%%%%%%%%

\section{Construction of Regions} \label{sec:Regions}

The construction of regions that we use to define the local formula $\mu$ is based on the choice of fundamental domains, not only for $\Lambda$, but for all occurring induced sublattices. Different fundamental domains form different regions and ultimately result in different local formulas.

By definition,
fundamental domains (for the lattice $\Lambda$ considered as an additive
group acting on $V$ by translation)
have the property that
every $\Lambda$-orbit of $V$ meets $T$ in exactly one point. Besides being bounded and connected, we further require our fundamental domains to contain $0$, which we obtain by translating an arbitrary fundamental domain by the negative of the unique lattice point it contains. 
We make this assumption to simplify notation considerably, but it can
actually be omitted, which we will briefly discuss in 
Section~\ref{sec:compsym} after Example~\ref{Ex:SquareNonStandard}.

\begin{Ex*}
An important family of examples of fundamental domains are \emph{Dirichlet--Voronoi cells}. %For lattices sometimes referred to as 'Wigner-Seitz cells' 
Given a space $V$ and an inner product $\langle \cdot  ,\cdot \rangle$ with induced norm $\| \cdot \|$, the Dirichlet--Voronoi cell of a sublattice $L \subseteq \Lambda$ is defined as 
\begin{equation*}
\DV(L, \langle\cdot,\cdot\rangle):=\{x\in \lin(L) : \|x\| \leq \|x-a\| \textnormal{ for all } a\in L\}. 
\end{equation*}
In this definition, it is not yet a fundamental domain of the lattice~$L$, since it is closed and thus translates by lattice points can intersect on the boundary.

However, by considering  the Dirichlet--Voronoi cell half open, it can be seen as a fundamental domain of the lattice. %We consider the group action of the lattice $\Lambda$ seen as an additive group acting on $V$ by translation. Given a closed Dirichlet--Voronoi cell $D$, we can construct a fundamental domain as follows: We take the equivalence classes modulo the group action of the realtive interior of each face of $D$, choose one representative out of each class and take the union.  
In Figure \ref{fig:Voronoi}, two different Dirichlet--Voronoi cells  in
$\R^2$ with lattice $\Z^2$ are given. They correspond to the standard
inner product and the inner product $\langle x,y\rangle = x^tGy$
defined by the Gram matrix $G=\begin{psmallmatrix}2&1\\1&2\end{psmallmatrix}$, respectively.

\begin{figure}[h]
\begin{center}
\scalebox{0.2}{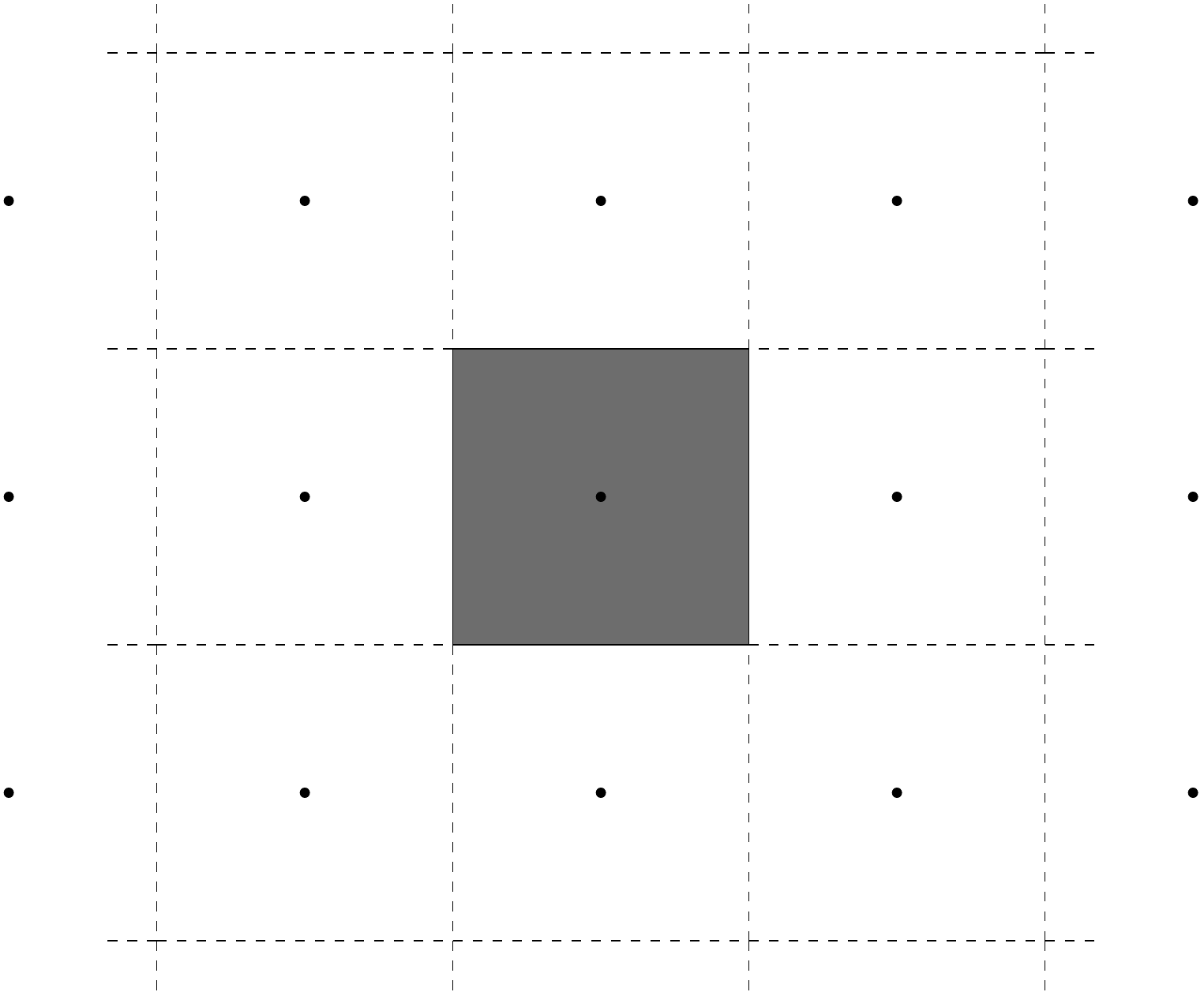}  \qquad \qquad \qquad
\scalebox{0.2}{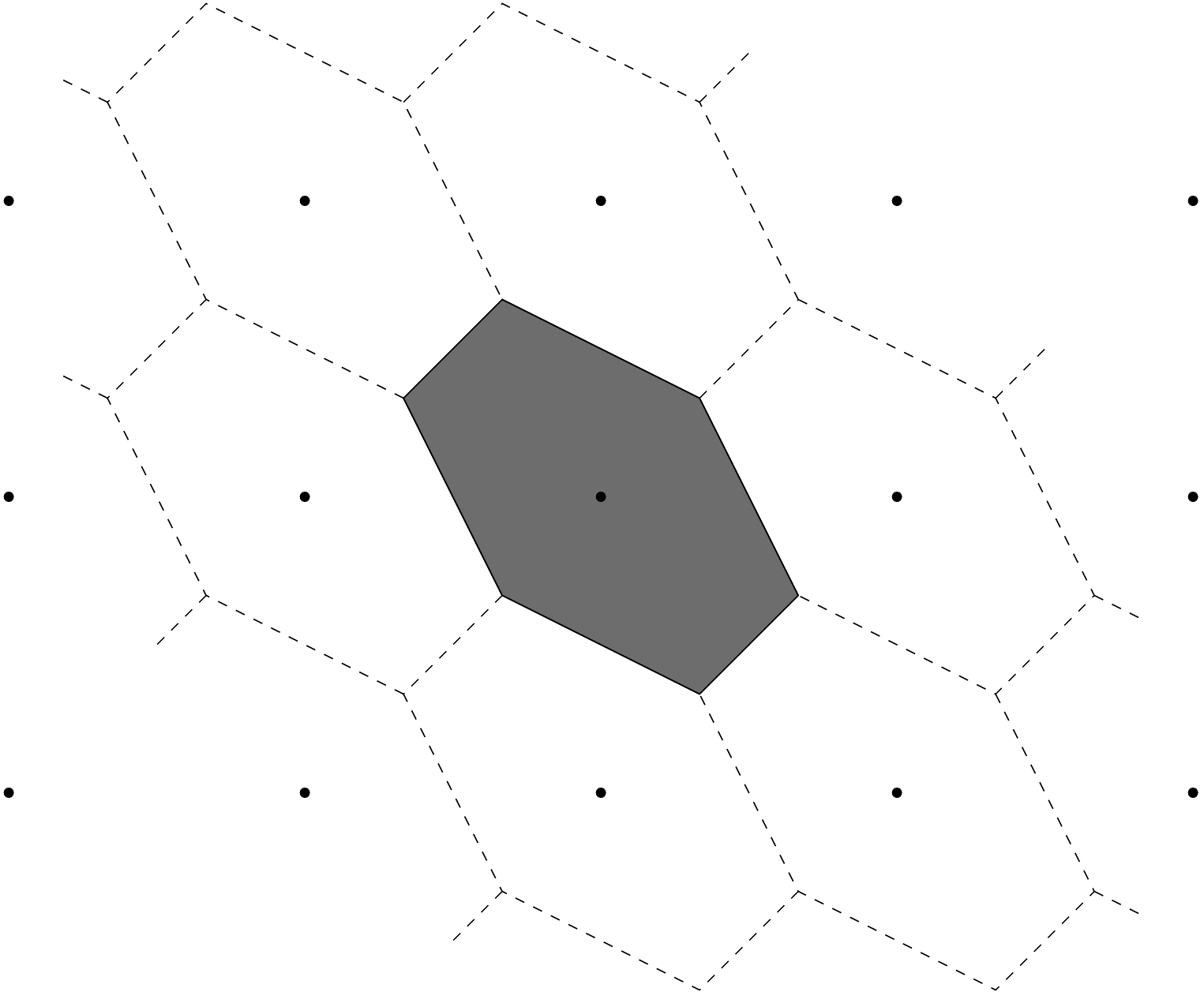}
\end{center}
\caption{Square and hexagonal Dirichlet--Voronoi cells of $\Z^2$.}
\label{fig:Voronoi}
\end{figure}

\end{Ex*}

For a pointed rational cone $C$ in $V$, let $L(C):=\Lambda \cap C^\perp$ be the induced lattice in  $C^\perp$ and let $T(C)\subseteq C^\perp$ be a fundamental domain of $L(C)$. % in $C^\perp$. 
If we consider the domain complex as illustrated in
Figure~\ref{fig:Tiling&DC} on the left, 
we see that the structure of the
domain complex is periodic with respect to lattice translations from~$L(C)$.
To obtain these periodicities, we build the regions as part of the strip $\lin(C)+T(C)$.  For each face $K$ of $C$, we further cut out all translates of regions $R(K)$ that 'fit properly'
 into $C^\vee$. A thorough definition of the inductive construction of the regions  is given below. The definition is somewhat technical. As an example, we give a detailed picture for two rational cones of different dimension in $\R^2$ in Figure \ref{fig:ConstructRC}.

\begin{Construction} 
%We first construct the regions for \emph{pointed} rational cones. The construction for genereal rational cones will then be reduced to that. Note that if a cone is pointed, then so are all its faces.
Let $C$ be a pointed rational cone in $V$. 
 If $C=C_0=\{0\}$ is the trivial cone, we set
\begin{equation*}
R(C_0):=T(C_0).
\end{equation*}

Otherwise, if $\dim(C)\geq 1$, we assume we have constructed all regions $R(K)$ for faces $K<C$. 
Let $X_K^C$ be the set of all points $x$ in $L(K)$ that satisfy:

\begin{compactenum}[(I)]
\item \label{Ppt:Inside}  $x+R(K) \subseteq \Int(M^\vee)$ \hfill for all rays $M \leq C$ with $M\nleq K$\; \;\,  and
\item \label{Ppt:Nonintersect} $\left( x+ R(K) \right) \cap \left(x'+ R(K')\right) = \varnothing$ \; \hfill  \; for all $K'<C$,   with
 \raggedleft{ $K'$ \textnormal{ incomparable to } $K$  and $x'\in L(K')$.}   
\end{compactenum}
Then  we define the region $R(C)$  of $C$ as
\begin{equation} \label{Eqn:region}
R(C):=\left(V\backslash  \bigcup_{K<C} \left(  X_K^C+R(K)\right) \right) \cap \left( T(C) + \lin(C)\right) \cap \upper {C^\vee}.
\end{equation}

\begin{figure}[h]
\begin{center}
\scalebox{0.8}{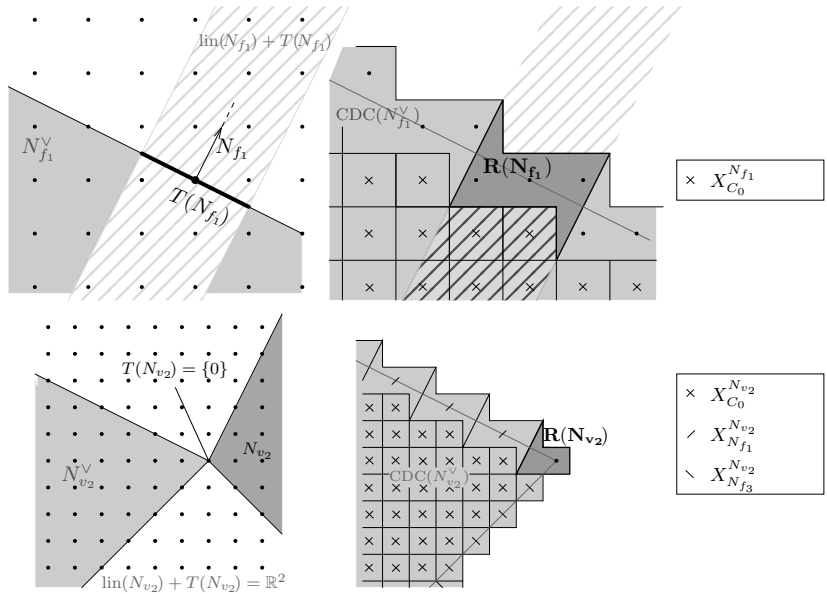} 
\end{center}
\caption{Construction of the regions $R(N_{f_1})$ and $R(N_{v_2})$ for the 1-dim. cone $N_{f_1}$ (above) and the 2-dim. cone $N_{v_2}$ (below). 
}
\label{fig:ConstructRC}
\end{figure}

\end{Construction}

Note that $R(C)$ is a piecewise linear set, respectively a union of
finitely many convex polyhedra (a {\em polyhedral complex}).
Therefore the values $v_C$ and $w^C_K$ (cf. \eqref{Eqn:DefDCVolume} and \eqref{Eqn:DefCorrectionVolume})
in the definition of~$\mu(C)$
in~\eqref{Eqn:muIntro} are obtained from volume computations of
finitely many convex polyhedra. In a sense, we hereby also obtain a
kind of geometric interpretation for the Ehrhart coefficients.

%%%%%%%%%%%%%%%%%%%%%%%%%%%%%%%%%%%%%%%%%%%%%%%%%%%%%%%%%%%%%%%%%%%%%%%%%%%%%%%%%
%%%%%%%%%%%%%%%%%%%%%%%%%%%%%%%%%%%%%%%%%%%%%%%%%%%%%%%%%%%%%%%%%%%%%%%%%%%%%%%%%

\section{Computations, symmetry and Examples} \label{sec:compsym}
In this section we will go into computational considerations, give some easy examples that illustrate our construction of  local formulas and show how symmetries can
be taken advantage of.  

\subsection*{Computation}

Our constructions rely heavily on the computation of volumes. Given either a vertex- or a halfspace-representation of the polytope, the theoretical complexity of the \emph{volume computation problem} is known to be $\#P$-hard, see  \cite{Dyer}. The complexity is unknown in the case that both, the vertex- and the halfspace-representation are provided (cf. \cite{Bueeler}). Nevertheless, for practical applications there are a lot of elaborated implementations for computing volumes using triangulations and signed decompositions, 
respectively as in \cite{Cohen}, \cite{Lasserre}, \cite{Lawrence}. 
For a practical realization of our construction it is also necessary  
to compute shortest vectors, cf. \cite{Micciancio}.
At this point, we have a working prototype implementation in Sage \cite{Sage}. 
By default it computes values of $\mu$ based on Dirichlet--Voronoi cells (definition see Section Symmetry below) 
and performs well up to dimension 4. 
The program does not yet fully utilize the optimized algorithms and is handling the operations on 
 non-convex regions in a rather naive way, such that there is a lot of room for improvement. 
As mentioned before, a beneficial property of our construction is that it does not rely on simplicial
(cf. \cite{PommersheimThomas}) or even unimodular triangulations of cones (cf.~\cite{BerlineVergne}).  
However, for the construction of %Pommersheim and Thomas \cite{PommersheimThomas} and 
Berline and Vergne \cite{BerlineVergne}, it can be shown that the $\mu$-values   
can be computed in polynomial time with respect to the dimension $d$ of the polytope as long as 
the codimension $\dim(P)-\dim(f)$ for the face $f$ is fixed (cf.~\cite{BarvinokBook}). 
As our construction is based on volume computations of $d$-dimensional polytopes we do not expect such a complexity result to hold.

\begin{figure}[h]
\begin{center}
\scalebox{0.38}{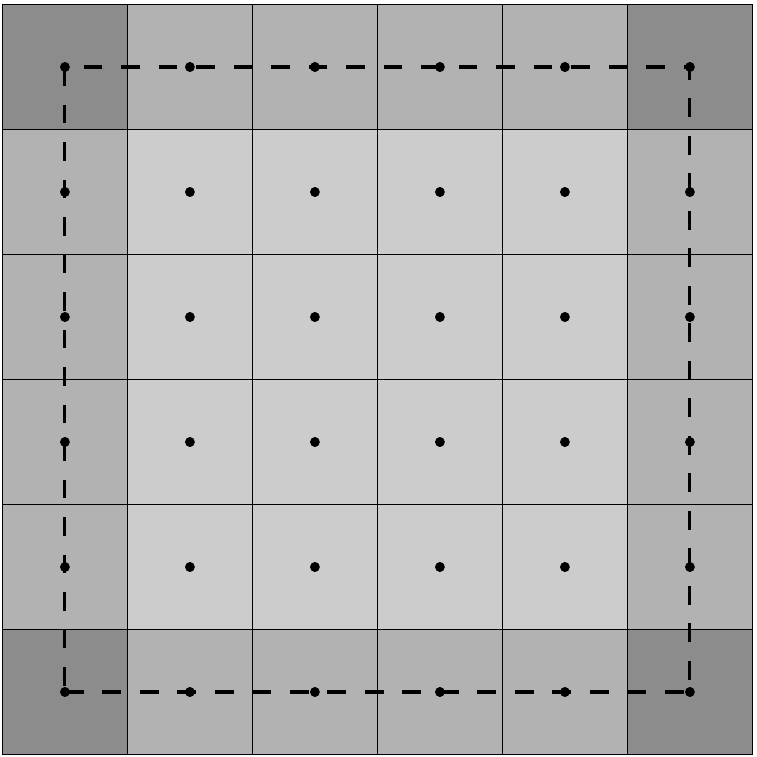}  \qquad 
\scalebox{0.38}{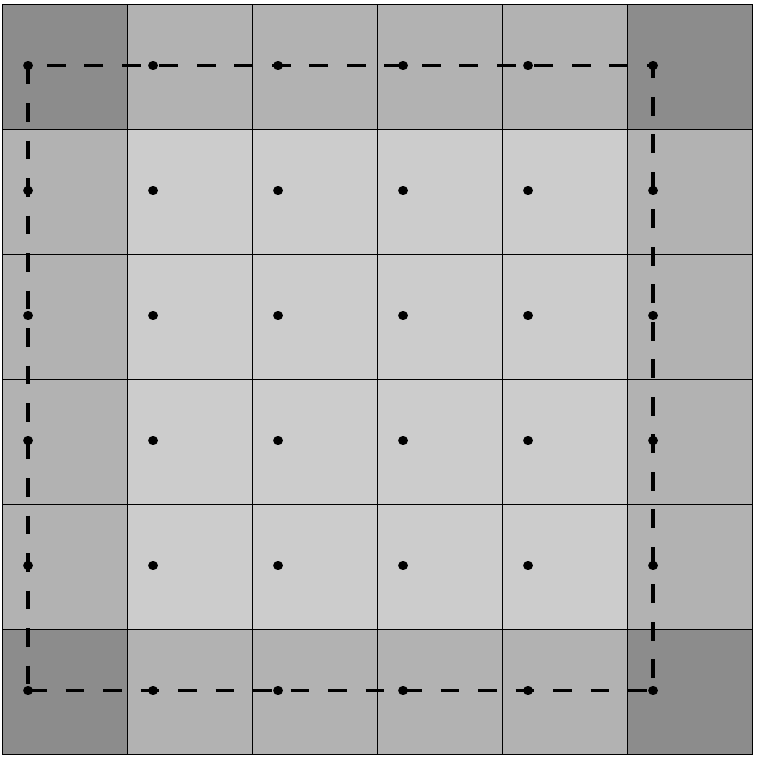}  \qquad 
\scalebox{0.2}{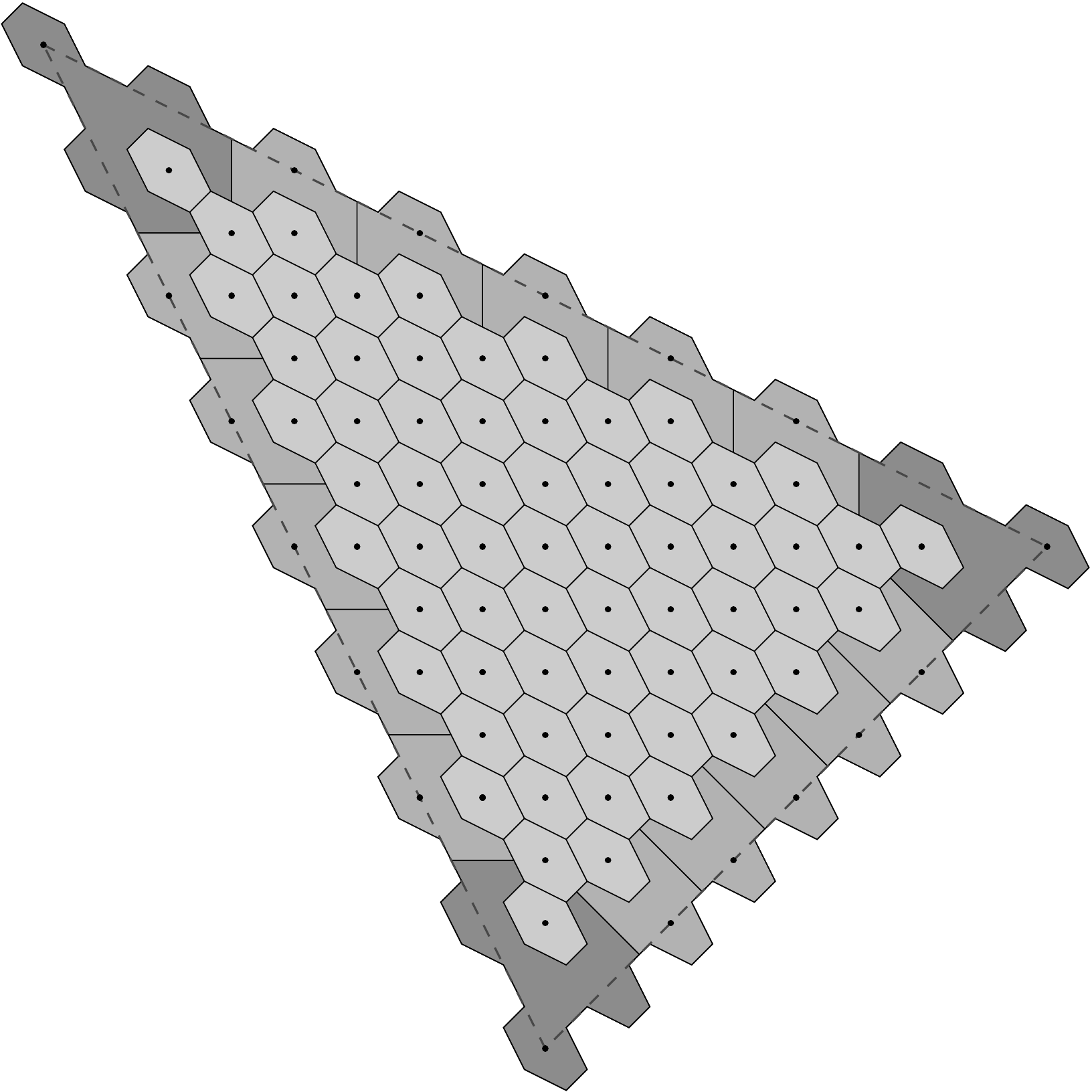}
\end{center}
\caption{Tiling by regions from Examples \ref{Ex:SquareStandard}, \ref{Ex:SquareNonStandard} and \ref{Ex:TriangleNonStandard}, respectively.}
\label{fig:Tilings}
\end{figure}

In the following examples, a lattice polytope $\mP$ is given as the convex hull of specific lattice points. But since translation by lattice points and dilation by a positive integer do not change the values of $\mu$, we will not show a coordinate system in our figures. In fact, since the tiling from Theorem \ref{Thm:TilingIntro} demands a certain dilation and since the structure of the tiling becomes clearer for larger dilations, in the following examples we  give all figures of $\mP$ dilated by a factor of at least~$3$. 

\begin{Ex} \label{Ex:SquareStandard}
Let $V=\R^2$, $\Lambda=\Z^2$ and
$\mP=H:=\conv\{(0,0),(1,0),(1,1),(0,1)\}$, the unit square. Let
$\langle\cdot , \cdot\rangle$ be the standard inner product.  We take
a  Dirichlet--Voronoi cell  as a fundamental domain of $\Z^2$ and also
of all sublattices. Then $D:=\DV(\Z^2, \langle\cdot , \cdot\rangle)$ is a square as shown on  the left of Figure~\ref{fig:Voronoi}  and for a one dimensional subspace $U\subseteq\R^2$, $D(U):=\DV(U\cap \Z^2, \langle\cdot , \cdot\rangle)$ is
the segment with vertices being midpoints 
between the origin and its two neighboring lattice points. From Theorem \ref{Thm:TilingIntro} we get  the regions as shown on the left of Figure \ref{fig:Tilings}. Since the symmetries of the Dirichlet--Voronoi cell  and of $H$ are the same, we get the same values for all faces with the same dimension for reasons that are discussed below. 
The values of $\mu$ for the normal cones $N_f$ of faces $f\leq H$  are then:
\begin{equation}
\begin{array}{ c | c c c c c c c  }			
\dim(f) & 2 & 1 & 0 \\
  \hline
  \mu(N_f) & 1 & \frac{1}{2}  & \frac{1}{4}  \\
\end{array}
\end{equation}
That yields the Ehrhart coefficients $e_2 = 1$, $e_1=  2$ and $e_0=1$.
\end{Ex}

\begin{Ex} \label{Ex:SquareNonStandard}
As above, let $V=\R^2$, $\Lambda=\Z^2$ and $\mP=H$ be the unit square. We again take the standard inner product, but as a fundamental domain we take $D_\eta$ instead of $D$, which is  a translate by $(\eta, 0)$ of the fundamental domain from the previous example, $D_\eta:=D+(\eta,0)$, where $\eta \in \R$ is any real number with $-\frac{1}{2}<\eta<\frac{1}{2}$. For the sublattice $\lin((0,1))\cap \Z^2$ we take the usual Dirichlet--Voronoi cell and for the sublattice $\lin((1,0)) \cap \Z^2$ we again take the translate of the Dirichlet--Voronoi cell by $(\eta,0)$. We then get a tiling by regions as shown in the middle of Figure \ref{fig:Tilings}. The resulting  (possible irrational) values for $\mu$ are: 
\begin{equation}
\begin{array}{ l || c | c | c | c | c   }			
  \textnormal{face } f & T & f_1 , f_2 & f_3 , f_4 & v_1 , v_2 & v_3 , v_4 \\
  \hline
  \mu(N_f) & 1 & \frac{1}{2}+\eta  & \frac{1}{2}-\eta  &  \frac{1}{4}+\frac{1}{2}\eta  &  \frac{1}{4}-\frac{1}{2}\eta  \\
\end{array}
\end{equation}
This example shows that irrational values are actually possible, showing a difference to all previous constructions of local formulas which have rational values only.
\begin{remark}
It is actually possible to drop the assumption made in Section~\ref{sec:Regions}, that 0 is contained in each fundamental domain. The parts that change in the proofs are mainly that  $\X(tf)$ is not necessarily contained in $tf$, but only in the affine span $\aff(tf)$ and the chosen radii in the proof of Lemma~\ref{Lm:bounded} might get larger, but are still finite.

The values for $\mu$ in Example~\ref{Ex:SquareNonStandard} are
therefore also applicable for $\eta \in \R$ arbitrary, which in particular allows the values to be negative even for the easy example of the unit square. 
\end{remark}

\end{Ex} 

\begin{Ex} \label{Ex:TriangleStandard}
Let $V=\R^2$, $\Lambda=\Z^2$ and $\mP=S$, the triangle (simplex) shown on the left of Figure~\ref{fig:Basics}. We can think of $S$ as the triangle that is the convex hull of the vertices $v_1=(1,0)$, $v_2=(2,1)$ and $v_3=(0,2)$.
On the right hand side of Figure~\ref{fig:Tiling&DC}, one can see the tiling of $4\cdot S$ by regions that we get if we choose the Dirichlet--Voronoi cell with respect to the standard inner product in $\R^2$.  Figures \ref{fig:vN} and \ref{fig:w} show some values for the \vDC and the \w which determine the values for~$\mu$:
\begin{equation}
\begin{array}{ l | c c c c c c c  }			
  \textnormal{face } f & S & f_1 & f_2 & f_3 & v_1 & v_2 & v_3 \\
  \hline
  \mu(N_f) & 1 & \frac{1}{2} &  \frac{1}{2} & \frac{1}{2} & \frac{3}{8} & \frac{3}{8} & \frac{1}{4}  \\
\end{array}
\end{equation}

\subsection*{Symmetry}
In Example \ref{Ex:SquareStandard} it was possible to use the standard inner  product and get the same values for each face in the same dimension. This principle can easily be generalized using suitable Dirichlet--Voronoi cells.

Let $\mP$ be a lattice polytope and $\G$ a subgroup of all lattice symmetries of $\mP$, i.e. $\G$ is a finite matrix group with $A\cdot \mP:=\{A\cdot x: x\in \mP\}=\mP$ and $A\cdot \Lambda = \Lambda$ for all $A\in \G$. Then we can define a $\G$-invariant inner product by taking 
\begin{align}
\langle x,y\rangle_\G :=x^t G y & & \textnormal{ for all } x,y\in V,
\end{align}
with the Gram matrix  $G$ given by
\begin{equation}
G:= \frac{1}{|\G|} \sum_{A\in \G} A^tA.
\end{equation}
Let $\|\cdot\|_\G$ be the induced norm and
let $D$ be the Dirichlet--Voronoi cell for $\Lambda$ given by the inner product,
\begin{equation*}
D:=\DV(\Lambda, \langle\cdot,\cdot\rangle_\G)=\{x\in V : \|x\|_\G \leq \|x-p\|_\G \textnormal{ for all } p\in \Lambda\}. 
\end{equation*}
\end{Ex}

Then $D$ is invariant under the action of $\G$:
Let $x\in D$, then for $A\in \G$ we have
\begin{align*}
\|Ax\|_\G = \|x\|_\G \leq \| x- p\|_\G = \| A x - A p\|_\G & &\textnormal{ for all } p \in \Lambda.
\end{align*}
Since $ A \Lambda = \Lambda$, we get $AD\subseteq D$ for all $A\in \G$. Substituting $A$ by $A^{-1}$, we get $A^{-1} D\subseteq  D $ which yields $ D\subseteq AD$ and hence $AD=D$.
Similarly, we see that for all faces $f$ in the same $\G$-orbit the
normal cones and Dirichlet--Voronoi cells in $\Lambda \cap N_f^\perp$
are mapped onto each other. Hence, the used regions are invariant under the action of $\G$ and $\mu$ is constant on $\G$-orbits.
%invariant inner product $->$ regions are mapped to regions, normal cones to normal cones and tadaaa: Same value for all of them -- for free  :-P

\begin{Ex} \label{Ex:TriangleNonStandard}
Again, let  $V=\R^2$, $\Lambda=\Z^2$ and $\mP=S$ as in the previous example. One might notice that $S$ is invariant under the action of the group $\G=\langle \begin{psmallmatrix} 0&1\\-1&-1\end{psmallmatrix}\rangle$ of order 3. 
An invariant  inner product  $\langle x,y\rangle_\G = x^tGy$ is defined by the Gram matrix $G=  \begin{psmallmatrix}2&1\\1&2\end{psmallmatrix}$. The Dirichlet--Voronoi cell corresponding to $\langle \cdot,\cdot \rangle_\G $ is the hexagon shown on the right of Figure \ref{fig:Voronoi} and the resulting tiling is given on the right of Figure \ref{fig:Tilings}. Since all faces of the same dimension lie in the same orbit, the computation of $\mu$ can be reduced:
\begin{equation*}
\begin{array}{ c | c c c   }			
\dim(f) & 2 & 1 & 0 \\
  \hline
  \mu(N_f) & 1 & \frac{1}{2}  & \frac{1}{3} 
\end{array}
\end{equation*}
\end{Ex}

%%%%%%%%%%%%%%%%%%%%%%%%%%%%%%%%%%%%%%%%%%%%%%%%%%%%%%%%%%%%%%%%%%%%%%%%%%%%%%%%%%%
%%%%%%%%%%%%%%%%%%%%%%%%%%%%%%%%%%%%%%%%%%%%%%%%%%%%%%%%%%%%%%%%%%%%%%%%%%%%%%%%%%%

\section{Proof of Theorem \ref{Thm:TilingIntro} (Tiling)} \label{sec:Tiling}

Given a full dimensional lattice polytope $\mP$, we can choose fundamental domains for all relevant sublattices of $\Lambda$ and construct all regions $R(N_f)$ for the normal cones $N_f$ of  faces $f\leq \mP$. 
We want to show that it is possible to take translated copies of the regions to form a tiling of  $\upper \mP$, 
where the region $R(N_f)$ is translated by lattice points $x$ in a dilation of the face $f$. The set $\X(tf) \subseteq tf \cap \Lambda$ of these lattice points in $tf$ for some sufficiently large integer $t\in \Zg$ is yet to be defined, it stands in strong relation to the sets $X^C_K$ that we used in the construction of regions in Section \ref{sec:Regions}.  
The aim of this section is to find the right definition of the sets $\X(tf)$ and to prove Theorem \ref{Thm:TilingIntro}, which we recall here:

\Tiling*

For an example of such a tiling, see the right of Figure \ref{fig:Tiling&DC}.  
An intermediate result of the proof is the following: If we start with just one pointed rational cone $C$, we get a tiling of  $\upper {C^\vee}$ by taking translates of the regions $R(K)$ for all $K\leq C$.    This result is given in Lemma \ref{Lm:TilingOneCone}. But before we start with that, we need another technical observation about a certain periodicity in the construction of the regions, namely Lemma \ref{Lm:Periodicity}. 

As in the preceding section, for a pointed rational cone $C$, we write $L(C):=\Lambda \cap C^\perp$ for the induced lattice in the orthogonal space of $C$ and let $T(C)$ be an arbitrary but fixed fundamental domain of $L(C)$.

\begin{lemma} \label{Lm:Periodicity}
Let $C$ be a pointed rational cone in $V$. Then for all $y\in L(C)$ we have
\begin{equation*}
y+\left(V\backslash \bigcup_{K<C} \left(  X_K^C+R(K)\right)\right) = V\backslash \bigcup_{K<C} \left(   X_K^C +R(K)\right).
\end{equation*}
In other words, the construction of the region $R(C)$ in Equation (\ref{Eqn:region}) is  invariant under translation of points in the lattice $L(C)$. 
\end{lemma}

\begin{proof}
Let $y\in L(C)$. Since translation by $y$ is a bijection, it commutes with unions and complements:
\begin{equation*}
y+\left(V\backslash \bigcup_{K<C} \left(  X_K^C +R(K)\right)\right)=\left(V\backslash \bigcup_{K<C} \left(   X_K^C +y+R(K)\right)\right)
\end{equation*} 
In order to prove the lemma, we only need to show that $  X_K^C$ is invariant under translation by $y$, i.e. $y+  X_K^C= X_K^C$ for all faces $K<C$.

Let $   K<C$ be a face of $C$ and $x\in    X_K^C=   X_K^C$, i.e. $x$ meets conditions (\ref{Ppt:Inside}) and  (\ref{Ppt:Nonintersect}). 
We want to show that $x+y\in X_K^C$. 
Since $y \in C^\perp$ and  $C^\perp + M^\vee = M^\vee$,  for all rays $M<C$, we have  
\begin{align*}
 x+R(K) \subseteq \Int(  M^\vee) \, \Leftrightarrow \, x+y+R(K) \subseteq \Int( M^\vee)
\end{align*}
which in particular  means that $x+y$ meets condition (\ref{Ppt:Inside}).

Now let $K'<C$ be a face of $C$ such that $K$ and $K'$ are incomparable and let $x'\in L(K')$ (cf. condition (\ref{Ppt:Nonintersect})). Since $y\in L(C)\subseteq L(K')$, we get:
\begin{align*}
\left(x+y+ R(K)   \right) \cap \left(x'+R(K')  \right) & =   \left( x+y+R(K)   \right) \cap \left(x'+y-y+R(K')  \right) \\
			& =y+  \left( \left(x+ R(K)    \right) \cap ( \underbrace{x'-y}_{\in L(K')} )+R(K')  \right) \\
			&= \; \varnothing.
\end{align*}
The last step follows from the fact that $x$ meets Condition (\ref{Ppt:Nonintersect}) and the whole equation shows that $x+y$ also meets condition (\ref{Ppt:Nonintersect}).
Altogether, we have shown that $x+y \in X_K^C$ for all $x\in  X_K^C$ and all $y\in L(C)$. That means $y+ X_K^C \subseteq  X_K^C$. Conversely, let $x\in  X_K^C$, then $x=x-y+y$. As with $y\in L(C)$ we also have $-y\in L(C)$, we can use that $x-y \in  X_K^C$ and hence  $x\in  y+X_K^C$. We thus get $ X_K^C \subseteq  y+X_K^C$ which finishes the proof.
\end{proof}

\begin{lemma} \label{Lm:TilingOneCone}
For any pointed  rational cone $C$ we have a tiling
\begin{align*}
\{x+R(C) : \  x\in L(C)\} \ \cup \ \{x+ R(K): \ K<C, \ x\in X^C_K\}
\end{align*} 
of  $\upper \mP$ consisting of lattice point translates of regions corresponding to $C$ and its faces.
\end{lemma}

\begin{proof}
In general, for a lattice $L$ in $V$ (not necessarily of full rank) and subsets $A, B\subseteq V$ with the properties that $B+L=V$ and $A+L=A$ we have that
\begin{equation}\label{Eqn:L+AB}
L+(A\cap B)=A.
\end{equation}
We show both inclusions:
\begin{align*}
\subseteq :\; \; \; & L+(A\cap B) \subseteq L+A=A. \\
\supseteq: \; \; \; &\nn{Let }  x\in A.  \nn{ Since } V=L+B, \nn{ we can write } x=l+b \nn{ with } l\in L \nn{ and } b\in B. \\
&\nn{Then } b=x-l\in A+L=A\nn{ and hence }  x\in L+(A\cap B).
\end{align*}

We have that  $C^\vee$ is invariant under $C^\perp$, thus so is $\upper {C^\vee}$.

From (\ref{Eqn:L+AB}) and Lemma \ref{Lm:Periodicity} we deduce

\begin{align} 
\begin{split} 
\label{Eqn:L+R}
  L(C)+R(C)  & =\, \underbrace{L(C)}_{L} + \left[ \underbrace{ \left( V \backslash \bigcup_{K < C}(  X_K^C+ R(K))\right)\cap \upper {C^\vee}}_{A} \cap 
\underbrace{ \left( T(C)+\nn{lin}(C)\right)}_{B} \right] \\
& = \; \;  \left( V \backslash \bigcup_{K < C}(  X_K^ C+ R(K)) \right)\cap \upper {C^\vee} .
\end{split}
\end{align}

Thus, we get

\begin{align*}
 \upper \mP  = \left( L(C)+R(C) \right) \cup \left( \bigcup_{K<C} \left(    X_K^C + R(K) \right)\cap \upper {C^\vee}\right).
\end{align*}

Since $R(C)\subseteq T(C)+\lin(C)$, we know that  the translates of $R(C)$ by points in $L(C)$ do not intersect. 
For each $K<C$, the set $  X_K^C$ is a subset of $L(K)$, so the same argument shows that the sets $\{x+R(K): \ x\in   X_K^C\}$ have pairwise empty intersections. 
Now we only need to show that for two faces $K, K'<N$ the sets of the form $x+R(K)$ and $y+R(K')$ with $x\in X_K^C$ and $y \in  X_{K'}^C$ do not intersect. 
If $K$ or $K'$ is a face of the other one, say $K'<K$, we have

\begin{align*}
X_K^C+R(K) 
		& \subseteq L(K) +R(K) \\
		& \overset{(\ref{Eqn:L+R})}{=} V \backslash \left(  \bigcup_{M < K}(  X_M^K+ R(M))\right) \\
		&  \subseteq V\backslash \left(  X_{K'}^K+R(K') \right)  \\
		& \subseteq V\backslash \left(X_{K'}^C+R(K')\right).
\end{align*}
The last   inclusion follows from the property that  $X^K_{K'} \subseteq  X^K_{K'}$  whenever $K'<K<C$. 
If $K$ and $K'$ are incomparable in the face lattice, then $X_K^C+ R(K)$ and $X_{K'}^C+R(K')$ do not intersect by construction of $X_K^C$, Property \ref{Ppt:Nonintersect}.
\end{proof}

For a pointed rational cone $C$, Lemma \ref{Lm:TilingOneCone} yields a tiling of  $\upper C$ into copies of translated regions $R(K)$, for $K\leq C$. 
In particular, given a full dimensional lattice polytope $\mP$, we have a tiling of  $\upper \mP$  for each normal cone $N_{v}$ with $v$ a vertex of $\mP$. 
In Figure \ref{fig:TriTilingOneCone}, these tilings are given for the triangle $S$ that is shown in Figure \ref{fig:Basics} and was formally introduced in Section \ref{sec:compsym}. 
Comparing these tilings to the tiling in Figure \ref{fig:Tiling&DC} on the right, which we want to construct for Theorem \ref{Thm:TilingIntro}, one might already get an idea of how to achieve this goal: 
For all vertices $v$ of $\mP$, we take the tilings from Lemma
\ref{Lm:TilingOneCone} applied to all $N_v$ and translate each by
$tv$. In this joint tiling, we disregard all translates of regions $R(N_f)$, with $f\leq \mP$ not a vertex, that do not fit.

\begin{figure}[h]
\begin{center}
\scalebox{0.78}{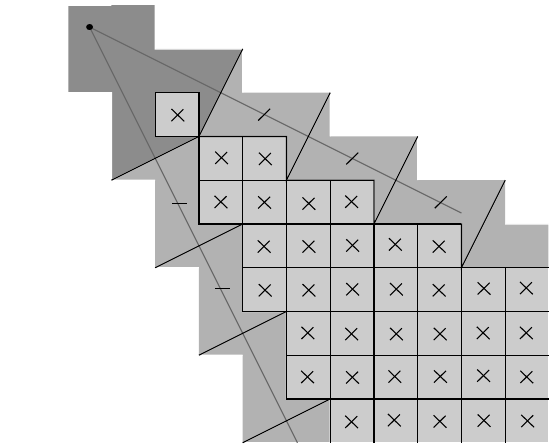} 
\hspace{-8pt}\scalebox{0.78}{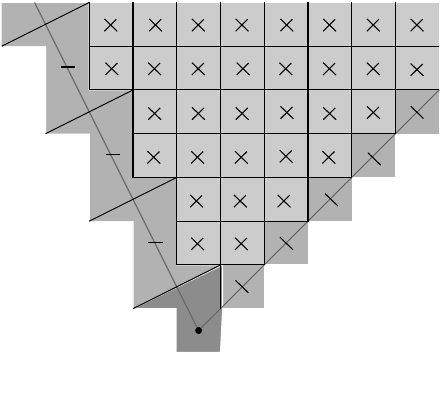} 
\hspace{6pt}\scalebox{0.78}{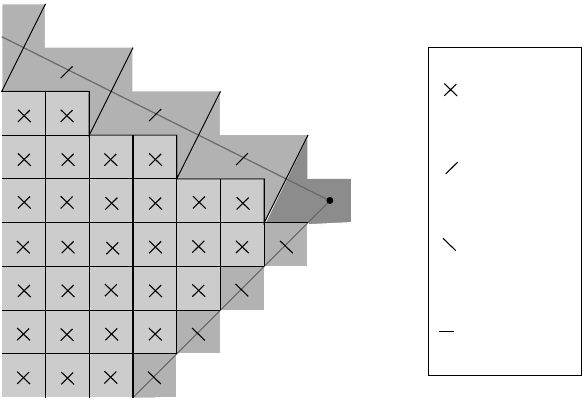} \\
\vspace{5pt}\hspace{8pt}\scalebox{0.78}{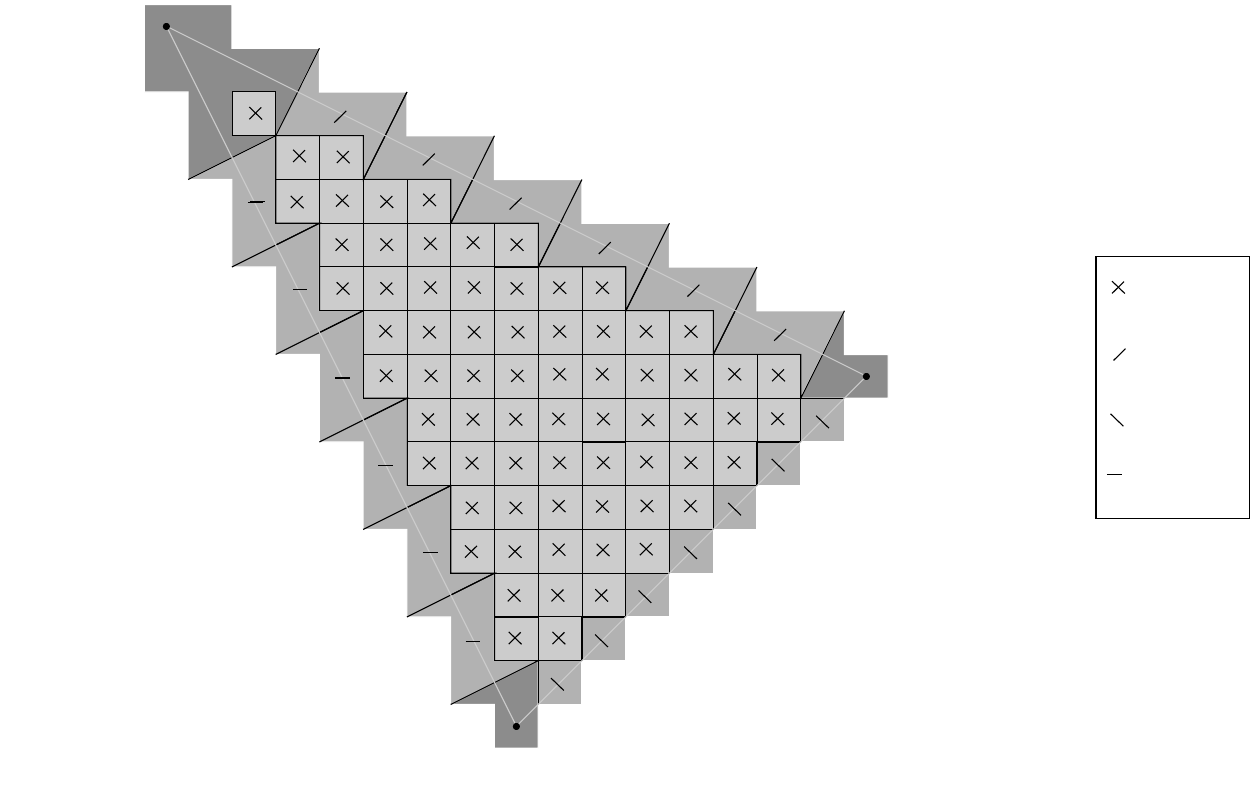} 
\end{center}
\caption{  Tiling for the triangle $tS$ according to Theorem \ref{Thm:TilingIntro}  as a connection of the tilings for its normal cones (Lemma \ref{Lm:TilingOneCone})  % (here, $t=4$)
}
\label{fig:TriTilingOneCone}
\end{figure}

Hence, for a face $f\leq \mP$, the correct way of defining $\X(tf)\subseteq \Lambda \cap tf$, the \emph{set of all feasible lattice points in $tf$}, is the following:
\begin{align} \label{Eqn:feasible}
\X(tf):=  \bigcap_{v \textnormal{ vertex of } f} X^{N_v}_{N_f} +tv.
\end{align}
By setting $X^{N_v}_{N_v}=\{0\}$ for a vertex $v$ of $\mP$, this definition is also valid for $\X(tv)$ and yields $\X(tv)=tv$ as desired. 

Before we start with the proof of Theorem \ref{Thm:TilingIntro}, we need the following lemma:

\begin{lemma} \label{Lm:bounded}
Let $C$ be a pointed rational cone in $V$ and $R(C)$ the corresponding region. Then   $R(C)$ is bounded. 
\end{lemma}

This property is very important for the proof of Theorem \ref{Thm:TilingIntro}, since without $R(N_{v_i})$ being bounded, there is no chance to fit the tilings from Lemma~\ref{Lm:TilingOneCone} into one tiling. Though the statement of Lemma \ref{Lm:bounded} might seem rather apparent, its proof is quite technical. 
Assuming Lemma \ref{Lm:bounded} for the moment, we now give the proof of Theorem~\ref{Thm:TilingIntro}.
To shorten notation, we will write $R(f)$ and $X^f_g$ instead of $R(N_f)$ and $X^{N_f}_{N_g}$, respectively, but keep in mind that these sets do not depend on the faces, but only the normal cones. 

\begin{proof}[Proof of Theorem \ref{Thm:TilingIntro}]
Let $\mP$ be a full-dimensional lattice polytope with vertices $v_1,\dots,v_m\in\Lambda$.
We start by specifying $t_0$.
Lemma \ref{Lm:bounded} yields that $R(f)$ is bounded for each $f \leq \mP$. Thus, for any $f,g \leq \mP$ that do not intersect, there is a $t_{fg} \in \mathbb{Z}_{>0}$ such that $(R(f)+t_{fg}\cdot f) \cap (R(g)+t_{fg}\cdot g) = \varnothing$.  We set
\begin{equation*}
t_0=\max\{t_{fg} : f,g\leq \mP \textnormal{ and } f\cap g = \varnothing\}. 
\end{equation*} 
This yields 
 that $(R(f)+t\cdot f) \cap (R(g)+t\cdot g) = \varnothing$ for non-intersecting $f,g \leq \mP$ and all $t\geq t_0$.

First, we show that the translated regions
\begin{align} \label{Eqn:Tiling}
\{x+R(f) : \; f\leq \mP, \, x\in \X(tf)\}
\end{align}
are pairwise disjoint, and in a second step we prove that they cover the whole space. 
Let $f,g  \leq \mP$ be arbitrary faces of $\mP$, let $t\in \Zgeq$ with $t\geq t_0$ and let $x\in  \X(tf)$ and $y\in \X(tg)$. 
For $f \cap g = \varnothing$ it follows from the construction of $t_0$ that $(x+R(f)) \cap (y+R(g)) = \varnothing$. 
Otherwise, we find $j\in \{ 1, \dots, m\}$ such that $v_j\in f \cap g$. Then, since $x\in \X(tf)$ and $y\in \X(tg)$, we have $(x-  t v_j)\in X^{v_j}_f$ and $(y-  t v_j)\in X^{v_j}_g$. By Lemma \ref{Lm:TilingOneCone} for $N_{v_j}$  the sets $(x-t v_j)+R(f)$ and $(y-t  v_j)+R(g)$ do not intersect, which then also holds for the translations  $ x+R(f)$ and $ y+R(g)$.

It remains to show that (\ref{Eqn:Tiling}) is indeed a covering of  $\upper \mP$ . 
To this end, let $p\in  \upper \mP $  be an arbitrary point. Lemma \ref{Lm:TilingOneCone} yields that
\begin{align*} \label{k0Tiling}
 \upper {N_{v_i}^\vee} &= \bigcup_{\substack{f \leq \mP \\ v_i \in f}} \left(  X^{v_i}_f + t v_i+ R(f) \right)
\end{align*}
 for each vertex $v_i$ of $\mP$. 
Hence, for each $i\in \{1, \dots, m\}$ we find $f_i  \leq \mP$ with $v_i \in f_i$ and $x_i\in X^{v_i}_{f_i}$ such that $p \in (x_i +  t v_i+ R(f_i))$. 

Let $v_i$ be a vertex, such that $f_i$ is smallest in dimension. Without loss of generality we can assume  $i=1$, hence,  $p \in (x_1 + tv_1 + R(f_1))$.
We want to show that  $x_1 +tv_1\in \X(tf_1)$, because then $x_1+tv_1+R(f_1)$ is an element of the set in (\ref{Eqn:Tiling}) and contains $p$.

 Let's assume this is not the case. After possibly renumbering, we can assume that $(x_1 +t v_1) \notin (X^{v_2}_{f_1}+t v_2)$.  In particular, we have $v_2 \in f_1$. 

But then we can find $f_2  \leq \mP$ with $v_2 \in f_2$ and  $x_2 \in X^{v_2}_{f_2}$ such that $p\in (x_2+t v_2+R(f_2))$. 
This yields 
\begin{equation*} 
 p\in  \left(  x_2 + t v_2+ R(f_2) \right) \cap \left(  x_1 + t v_1+ R(f_1) \right),
\end{equation*}
and thus
\begin{align} \label{f1f2}
\left(  x_2 + R(f_2) \right) \cap ( \underbrace{ x_1 + t v_1 -  t v_2}_{\in L(f_1)}+ R(f_1) ) \neq \varnothing,
\end{align}
which contradicts $x_2\in  X^{v_2}_{f_2}$ by property (\ref{Ppt:Nonintersect}), unless $f_1$ and $f_2$ are comparable. 

The case $f_2 \subsetneq f_1$ is not possible, since $\dim(f_2)\geq \dim(f_1)$ by assumption on the minimality of the dimension of $f_1$. The case $f_1 = f_2$ is not possible either, since $x_2+tv_2+R(f_1)$ and $x_1+tv_1+R(f_1)$ can only intersect if $x_2+tv_2=x_1+tv_1$, in which case $x_1+tv_1\in (X^{v_2}_{f_2}+tv_2)$.
%which we assumed not to be the case$.

We are left with the case $f_1 \subseteq f_2$ to be excluded. We now can consider $X^{f_1}_{f_2}$ and have the inclusion $X^{v_2}_{f_2} \subseteq X^{f_1}_{f_2}$ (since we only add conditions when going from $X^{f_1}_{f_2} $ to $ X^{v_2}_{f_2}$). Since $x_2 \in X^{v_2}_{f_2}$, we also have $x_2 \in X^{f_1}_{f_2}$. Then 
the sets $ (x_1 + t v_1 -  t v_2+R(f_1))$ and $ \left(  x_2 + R(f_2) \right)$ are part of the tiling that we get by applying Lemma \ref{Lm:TilingOneCone}  to $N_{f_1}$. But,  as we see in equation (\ref{f1f2}), the two sets do intersect, which is a contradiction.
%Via exclusion we have thus shown that $p$ is covered by the tiling given in (\ref{Eqn:Tiling})  and since $p\in V$ was arbitrary, this proves that (\ref{Eqn:Tiling}) is also a covering and hence a tiling of~$V$.

\end{proof}

For proving Lemma \ref{Lm:bounded}, we use the following linear algebra fact by  which we can assume a certain distance between points on faces of a cone. Its proof is straightforward using standard arguments and is omitted here.

\begin{lemma}\label{Lm:Distance}
Let $V_1, V_2\subseteq V$ be subspaces of $V$ with intersection $U:=V_1\cap V_2$. Then for all $r>0$ there exists $\delta>0$ such that for all $x\in V_1$
\begin{align*}
\dist(x,U)
> \delta \;\;\Rightarrow \;\; \dist(x,V_2)
>r.
\end{align*}
\end{lemma}

We prove Lemma \ref{Lm:bounded}  inductively. The idea is  easy: The regions are constructed by cutting something off. 
So a rather obvious approach to show that $R(C)$ is bounded is to show that every point in $V$ that is far enough away from the oririn is contained in a region that is in the complement of $R(C)$. The complement consists of all regions corresponding to faces $K<C$ and translated by  a lattice point in $L(K)$ that fulfills the properties (\ref{Ppt:Inside}) and (\ref{Ppt:Nonintersect}) for being in $X^C_K$.  
Therefore, we consider two cases, the first being that a point is away from the boundary of $C^\vee$. Then it will be easy to see that there is a fundamental domain $T(C_0)=R(C_0)$ of $\Lambda$ whose translate contains the point and is in the complement of $R(C)$. 
If a point $p$ is close to  the boundary and far enough away from $C^\perp$, we can show that close to that point we find a lattice point $x$ on the boundary and a region $R(K)$, $K<C$, such that $x+R(K)$ is in the complement of $R(C)$. We then face the obstacle that it is still unclear whether $p$ is in that particular region, since our regions are not even convex. 
The solution is to ensure that not only $x+R(K)$ is in the complement, but also all regions that cover the area around it.

\begin{proof}[Proof of Lemma \ref{Lm:bounded}] 
Let $C$ be a rational pointed cone. We prove the lemma by induction on
$n:=\dim(C)$.
We want to show that there exists a certain bounded set that contains $R(C)$. 
The bounded sets we want to consider are cylinders around the linear space $C^\perp$ with radius $r_C\in \Rg$: 
\begin{align*}
\C(r_C, C):= \left(C^\perp + B_{r_C}(0)\right) \cap \left( T(C) + \lin(C) \right),
\end{align*} 
where $B_r(x)$ is the open ball around $x\in V$ with radius $r\in \Rg$. Since $T(C)\subseteq C^\perp$, we can alternatively describe $\C(r_C,C)$ as 
\begin{equation*}
\C(r_C, C)=T(C)+(B_{r_C}(0) \cap \lin(C)).
\end{equation*} 

The case $n=0$ is simply noticing that $R(C_0)$ is a fundamental domain of the lattice $\Lambda=L(C_0)$ and as such it is by definition bounded. 

For $n>0$, let $C\subseteq V$ be a pointed rational cone of dimension $n>0$ and we assume that $R(K)\subseteq \C(r_K, K)$ for all faces $K<C$ with suitable $r_K\in \Rg$.

By construction, $R(C) \subseteq (T(C)+\lin(C))$. What we now need to show is that  $R(C) \subseteq \left(C^\perp + B_{r_C}(0)\right)$ for some $r_C\in \Rg$. We choose $r_C$  defined by the following constraints that might seem technical, but it has the advantageof being constructive.

\subsection*{Construction of $r_C$:}
\begin{quotation}
For each face $K<C$ we have  $R(K) \subseteq \C(r_K, K)$ by the inductive hypothesis. We define a second radius $r'_K$ by
\begin{displaymath}
r_K'=r_K+2 \cdot \max_{M<K} r_M.
\end{displaymath}
We recall that by Lemma \ref{Lm:TilingOneCone} we have a tiling of  $\upper {K^\vee}$ such that
\begin{align*}
   \upper {K^\vee}= \left( L(K)+R(K) \right) \cup \bigcup_{M<K} \left(  X_M^K + R(M) \right).
\end{align*}
So there exists an $s\in \Rg$ such that
\begin{equation}  \label{Eqn:U(K)}
U(K):= R(K) \cup \bigcup_{M<K} \; \bigcup_{x\in X^K_M \cap B_s(0) } x+R(M)
\end{equation} 
contains $\C(r_K', K)\cap  \upper {K^\vee}$. Since $U(K)$ is a finite union of bounded sets, $U(K)$ itself is bounded and we can find $u_K\in \Rg$ such that $U(K)\subseteq B_{u_K}(0)$.

Then $K_1, K_2<C$ with  $K_1 \vee K_2 =C$  implies $K_1^\perp \cap~K_2^\perp = C^\perp$ and we can apply Lemma \ref{Lm:Distance}: We find $\lambda_C\in \Rg$ such that $\dist(x,K_2^\perp)>u_{K_1}+u_{K_2}$ for all $x\in K_1^\perp$ with $\dist(x,C^\perp)>\lambda_C$. Let $t_K\in \Rg$ be a radius such that the fundamental domain $T(K)\subset B_{t_K}(0)$ and set
\begin{equation} \label{Eqn:hC}
h_C:=\lambda_C + \max_{K<C} \{2t_K\}.
\end{equation}
Then we  define 
\begin{equation} \label{Eqn:rC}
r_C:= \max_{K<C} \sqrt{h_C^2+ (r_K')^2}.
\end{equation}
\end{quotation}

To show that $R(C)\subseteq \left(C^\perp + B_{r_C}(0)\right)$, we show that for each point $p\in  V$ with $\dist(p, C^\perp)>r_C$ we  have either $p\notin \upper {C^\vee}$ or there are a face $K<C$ and a lattice point $x\in   X^C_K$ such that $p\in ( x+R(K))\subseteq \compl{R(C)}$. %\subseteq \compl{\big(\PR{C}\big) }$.

So let $p\in  \upper {C^\vee}$ with $\dist(p, C^\perp) \geq r_C$. We consider two cases. Figuratively speaking, we consider the case of $p$ being far away from the boundary of $C^\vee$ (in which case we can just find a translate of $R(C_0)$ that covers it), and the case $p$ being close to the boundary (where we use a translate of $U(K)$ as defined in (\ref{Eqn:U(K)}) to cover it by a translated region).

\subsection*{Case 1:} $\forall K<C$ with $\dim(K)\geq 1$, we have $\dist(p,K^\perp)\geq r_{K}'$.

Let $x\in \Lambda$ with $p\in x+R(C_0)$. Since $R(C_0)=T(C_0) \subseteq B_{r_{C_0}}(0)$, we get
\begin{align*}
\dist(x+T(C_0),K^\perp)\geq r_K'-2r_{C_0} \geq r_K >0 
\end{align*} 
for all $ K<C$ with $\dim(K)\geq 1$.
Hence, $(x+T(C_0))  \cap \bd(C^\vee)=\varnothing$. If $x+T(C_0) \subseteq V\backslash C^\vee $, we have $x+T(C_0) \subseteq V\backslash \upper {C^\vee}$. If $x+T(C_0) \in C^\vee$, then $x\in X^C_{C_0}$ and $p \in (x+R(C_0))  \subseteq \compl{R(C)}$.

 \subsection*{Case 2:} $\exists K<C$, $\dim(K)\geq 1$ with  $\dist(p,K^\perp)<r_K'$.

Let $K<C$ be the face of $C$ with maximal dimension such that $\dist(p,K^\perp)<r_K'$. Define $y:=p|_{K^\perp}$ to be the orthogonal projection of $p$ onto $K^\perp$ and let $x\in L(K)$ with $y\in x+T(K)$. 

As \textbf{first step} we want to show that $x\in X^C_K$. 

From $\dist(p,K^\perp)<r_K'$ by the Pythagorean theorem we can deduce
\begin{align*}
 \dist(y,C^\perp)^2+\dist(p,K^\perp)^2 = \dist(p,C^\perp)^2
\end{align*}
and thus
\begin{align*}
  \dist(y,C^\perp)^2 & = \dist(p,C^\perp)^2 - \dist(p,K^\perp)^2  \\
	& >r_C^2 - r_K' 
> h_C^2 + (r_K')^2 -(r_K')^2= h_C^2.
\end{align*} 
The last line follows from the premise that $\dist(p, C^\perp) \geq r_C$ and from the definition of $r_C$ in Equation (\ref{Eqn:rC}).
Since we have now shown that $ \dist(y,C^\perp) > h_C$, we have by definition of $h_C$ in (\ref{Eqn:hC}) that $\dist(x,C^\perp)>\lambda_C$ and thus by definition of $\lambda_C$
\begin{equation} \label{Eqn:1.uKuM}
\dist(x, M^\perp) > u_K + u_M
\end{equation}
for all $ M<C$  with $K\vee M =C$. 

We now want to show the same result holds
for all $M<C$ such that $K$ and $M$ are incomparable and $K \vee M < C$. By maximality of $K$ we get 
\begin{equation*}
\dist(p,(K \vee M)^\perp)>r_{K \vee M }'>r_{K \vee M }.
\end{equation*}

With exactly the same computations as above (with $K \vee M$ instead of $C$) we get 
\begin{equation}
\begin{aligned} \label{Eqn:2.uKuM}
 & \dist(y,(K \vee M)^\perp)^2 > h_{K \vee M} \\
 \Rightarrow \;\; & \dist(x,(K \vee M)^\perp) > \lambda_{K \vee M} \\
 \Rightarrow \;\; & \dist(x,M^\perp)>u_K+u_{M}.
\end{aligned}
\end{equation} 
That means
\begin{equation} \label{Eqn:uKuM}
\dist(x, M^\perp) > u_K + u_M
\end{equation}
for all $M<C$ incomparable to $K$.
Since $u_M>0$ and $x+R(K)\subseteq B_{u_K}(x)$ 
(in other words $x+R(K)$ has a positive distance to all other faces of $C^\vee$),  
we immediately get that $x$ fulfills Property (\ref{Ppt:Inside}) for being in $X^C_K$: 
 $x+R(K) \subseteq \Int (M^\vee)$ for all rays $M<C$ that are no rays of $K$. 

Regarding property (\ref{Ppt:Nonintersect}) of $X^C_K$, from Equation \eqref{Eqn:uKuM} we get 
\begin{equation*}
 (x+R(K)) \cap (x'+R(K')) = \varnothing
\end{equation*}
  for all $K'<C$ incomparable to $K$ and all $x'\in L(K')$.
%Since $(x+R(K))\backslash (x+\PRb{K})\subseteq x+T(K)+K$, as well as   $(z+R(M))\backslash (z+\PRb{M})\subseteq z+T(M)+M$ and $K, M$ are both normal cones of different faces of $C^\vee$, we also have
%\begin{equation*}
%(x+R(K)) \cap (z+R( M)) = \varnothing
%\end{equation*}
% for all $z \in L( M)$. 
Hence, $x$ also has property (\ref{Ppt:Nonintersect}) for being in $X^C_K$ and we have $x\in X^C_K$, which finishes step one.

Now $x+U(K)$ covers $x+(\C(r_K', K)\cap  \upper {K^\vee})$, where the latter contains $p$. 
If $p\in (x+R(K))$, we are done, since we have just shown that $x\in X^C_K$ and hence $(x+R(K))\subseteq \compl {R(C)}$. 
Otherwise, we have $p\in (a+R(K_1))$ for some $K_1<K$ and $a \in X^K_{K_1}$. 

The \textbf{second step} is to show that then $a \in X^C_{K_1}$ which yields $(a+R(K_1)) \subseteq \compl {R(C)}$.

 Since $a+R(K_1)\subseteq x+B_{u_K}$ and by Equation (\ref{Eqn:uKuM}) the distance of $x$ to $M^\perp$  in particular for all rays $M<C$ that are no rays of $K$, we see that $a$ has Property \ref{Ppt:Inside} for being in $X^C_{K_1}$.

Again, we need to consider different cases. Firstly, we observe that 
\begin{equation*}
\big(a+R(K_1)\big)\cap \big(b+R(M)\big) = \varnothing
\end{equation*}
 for all $M$ incomparable to $K$ (both for $K \vee M < C$ and $K \vee M = C$) and all $b\in L(M)$: 
We have $(a+R(K_1)) \subseteq B_{u_K}(x)$ and $(b+R(M)) \subseteq B_{u_{M}}(b)$ and as we have seen in (\ref{Eqn:uKuM}), we  have $\dist(x,M^\perp))>u_K+u_{M}$.

Secondly, we are left with the case $M<K$ and $M, K_1$ incomparable. 
But since we have $a\in X^K_{K_1}$, we get from property (\ref{Ppt:Nonintersect}) that
 $(a+R(K_1))\cap (b+(M)) = \varnothing$ for all $M<K$ with $M, K_1$ incomparable and all $b\in L(M)$. 
Hence, $a\in X^C_{K_1}$ and $p \in (a+R(K_1))\subseteq \compl R(C)$ as we wanted to show. 

Hence, we have shown that $R(C)$ is bounded.
\end{proof}

%%%%%%%%%%%%%%%%%%%%%%%%%%%%%%%%%%%%%%%%%%%%%%%%%%%%%%%%%%%%%%%%%%%%%%%%%%%%%%%%%
%%%%%%%%%%%%%%%%%%%%%%%%%%%%%%%%%%%%%%%%%%%%%%%%%%%%%%%%%%%%%%%%%%%%%%%%%%%%%%%%%

\section{Proof of Theorem \ref{Thm:MuIntro} (Local formula)} \label{sec:ProofLocForm}

 Assume we have chosen  fixed fundamental domains for all sublattices $L\subseteq \Lambda$. We recall the definition of the function~$\mu$ on rational cones that was given in Section~ \ref{sec:Intro}.  
We first set 
\begin{equation} \label{Eqn:muC0}
\mu(C_0):=v_{C_0}=1
\end{equation}
for the trivial cone $C_0=\{0\}$. 
For a pointed rational cone $C\subseteq V$ with $\dim(C)\!\geq\! 1$
we then define by induction on the dimension
\begin{equation} \label{Eqn:mu}
\mu(C):=v_C - \sum_{K< C} w^C_K \cdot \mu(K).
\end{equation}
Here, $v_C$ is the \vDC defined in~\eqref{Eqn:DefDCVolume}
and $w^C_K$ is the \w from~\eqref{Eqn:DefCorrectionVolume}.

For a rational cone $C\subseteq V$ that is not pointed but contains a maximal nontrivial linear subspace $U$, we can consider the pointed cone $C':=C\cap U^\perp$  in $U^\perp$, where we consider $U^\perp$ as a Euclidean space equipped with the induced inner product and the  lattice $\Lambda\cap U^\perp$. We can then construct $R(C') \subseteq U^\perp$ and set
 \begin{equation*}
\mu(C):=\mu(C').
\end{equation*} 

\addtocounter{thm}{-1}
\begin{thm}
The function~$\mu$ on 
rational cones in $V$ as defined in \eqref{Eqn:muC0} and \eqref{Eqn:mu}
is a local formula for Ehrhart coefficients. 
\end{thm}
That is, for every lattice polytope $\mP$ with Ehrhart polynomial $\E (t)=e_d t^d+e_{d-1}t^{d-1}+\dots + e_1t+e_0$, $t\in \Zgeq$, we have
\begin{align*}
e_i=\sum_{\substack{f \leq \mP \\ \dim(f)=i }} \mu(N_f) \vol(f), 
\end{align*}
for all $i\in\{0,\dots, d\}$.

As discussed in Section \ref{sec:Intro}, this is equivalent to 
\begin{equation*}
|\Lambda \cap t\mP| = \sum_{f\leq \mP} \mu(N_f)\vol(tf)
\end{equation*}
for all lattice polytopes $\mP$.

Before we prove Theorem \ref{Thm:MuIntro}, we make some simplifications. Since neither the Ehrhart polynomial, nor the function $\mu$, nor the relative volumes of faces change when we translate $\mP$ by a lattice point, we can without loss of generality assume that $0\in \mP$. %e.g. by translating it by the negative of a vertex of $\mP$. 
%Note that we are not assuming that 0 is in the interior of $\mP$, which might not be possible, since not every lattice polytope has an inner lattice point. 
Then $\lin(\mP)=\aff(\mP)$ and each normal cone $N_f$ with $f$ a face of $\mP$ contains the orthogonal space $\mP^\perp$ as a maximal  linear subspace. The definition thus yields $\mu(N_f)=\mu(N_f\cap\lin(\mP))$ for all faces $f$ of $\mP$. By considering $\mP \subseteq \lin(\mP)$, we can hence assume, without loss of generality, that $\mP$ is full dimensional and that all normal cones are pointed. 
As before, to shorten notation, for faces $f\leq \mP$ we henceforth write $R(f)$ instead of $R(N_f)$, also $L(f)$ instead of $L(N_f)$ for the sublattice $\Lambda\cap N_f^\perp\subseteq \Lambda$  and $T(f)$ for the fundamental domain $T(N_f)$ in $L(f)$. We also write $v_f$ and $w_f^g$ instead of $v_{N_f}$ and $w^{N_g}_{N_f}$ for faces $g<f\leq \mP$.
 But we keep in mind that these objects do not depend on the face $f$ itself, but only on the normal cone $N_f$. Note that $N_f\leq N_g$ if and only if $g\leq f$.

\begin{proof}[Proof of Theorem \ref{Thm:MuIntro}] To make the structure of the proof easier to grasp, we delay some steps into lemmas, which we will state and prove afterwards. 

Let $\mP$ be a full dimensional lattice polytope. Recall that $T$ is a fundamental domain of $\Lambda$. 
Since the relative volume is normalized, such that every fundamental domain has volume 1, we have the following equation for every $t\in \Zgeq$:
\begin{equation} \label{Eqn:Cellvolume}
| t\mP \cap \Lambda | =  \vol(\DC {t\mP})  = \vol((t\mP \cap \Lambda ) + T).
\end{equation}
Instead of counting the (discrete) number of lattice points in $t\mP$, we thus can compute the (continuous) volume of fundamental domains around each lattice point in $t\mP$. Following the notation of Section \ref{sec:Intro}, the right hand side of Equation~(\ref{Eqn:Cellvolume}) is the volume of the domain complex  of $t\mP$.

Let $\mathcal{P}$ be a polytope 
and $t\in \Zg$ big enough, such that we have a tiling of  $\upper \mP$ by regions as in  Theorem \ref{Thm:TilingIntro}:
\begin{equation*}
\{x+R(N_f) : \; f\leq \mP, \, x\in \X(tf)\},
\end{equation*}
with $\X(tf)$ the set of feasible lattice points in $tf$ as defined in  Section~\ref{sec:Tiling}, Equation~(\ref{Eqn:feasible}).
Then, as we will show in Lemma \ref{Lm:PNvee} below, we can divide the volume of the domain complex  into the parts in each region, which equals $v_f$ (the DC-volume in $R(F)$): 
\begin{displaymath}
 \vol( \DC {t\mP})= \sum_{f  \leq \mP} \left| \X(tf) \right|  \cdot v_{f}.
\end{displaymath}

Thus with the definition of $\mu(N_f)$ solved for $v_f$ we get
\begin{align*}
| \Lambda \cap t\mP |  
		& =  \sum_{f  \leq \mP} \left| \X(tf) \right|  \cdot v_f  \\ 
	& = \sum_{f \leq \mP} \left[    \left| \X(tf) \right| \cdot \left(   \mu(f) + \sum_{h>f}  w^f_h \cdot \mu(h)          \right)  \right]. \\ 
\end{align*}
We can now expand  the product and combinatorially rearrange the sum to get
\begin{align*}
| \Lambda \cap t\mP | 	& =\sum_{f \leq \mP} \left[   \left| \X(tf) \right|  \cdot  \mu(f) +  \left| \X(tf) \right| \cdot \sum_{h>f}  w^f_h \cdot \mu(h)    \right] \\ 
	& = \sum_{f \leq \mP} \underbrace{\left[   \sum_{g\leq f}  \left| \X(tg) \right|\cdot w^g_f  	 \right]}_{=:V(tf)} \cdot \mu(f) .
\end{align*}
In the last line  the expression $w^f_f$ for the \w  technically has not been defined yet --- we simply set $w^f_f:=1$ for faces $f\leq \mP$. 
%If one makes the effort to compare it to the definition and uses Lemma~\ref{Lm:FundInReg}, one notices that this is actually extends the definition consistently. 

By Lemma \ref{Lm:Vol(tf)} below we can now use that
\begin{equation*}
V(tf)=\vol(tf),
\end{equation*}
which  yields
\begin{equation} \label{Eqn:EhrTBig}
|  t\mP \cap \Lambda | =  \sum_{f  \leq \mP}  \vol(tf) \cdot \mu(f)
\end{equation}
for all $t\in \Zgeq$ with $t>t_0$ for a certain $t_0\in \Zgeq$. By Ehrhart's Theorem~\cite{Ehrhart}, we know that $\E(t)=|  t\mP \cap \Lambda |$ is a polynomial in $t$, as is the right hand side of Equation~(\ref{Eqn:EhrTBig}). Since these polynomials agree for infinitely many $t$, we have equality and get
\begin{equation*}
\E(t)= \sum_{f \leq \mP} \mu(N_f) \vol(tf)
\end{equation*}
for all $t\in \Zgeq$. That shows that for each choice of fundamental domains, the resulting function $\mu$ is a local formula for Ehrhart coefficients. 

\end{proof}

%As before, to shorten notation, we henceforth write $R(f)$ instead of $R(N_f)$ for faces $f\leq \mP$, but keep in mind, that it does not depend on the face $f$, but only on the normal cone $N_f$.

\begin{lemma} \label{Lm:PNvee}
We have 
\begin{displaymath}
 \vol( \DC {t\mP })= \sum_{f  \leq \mP} \left| \X(tf) \right|  \cdot v_f 
\end{displaymath}
for all $t\in \Zgeq$ big enough in the sense of Theorem \ref{Thm:TilingIntro}.
\end{lemma}

\begin{proof}
We recall the definition of $\X(tf)$ as
\begin{align*}
\X(tf):= \bigcap_{v \textnormal{ vertex of } f} X^{N_v}_{N_f} +tv,
\end{align*} 
where $ X^{N_v}_{N_f}$ is the set of lattice  points that we constructed in Section~\ref{sec:Regions}. For an illustration of the sets $ X^{N_v}_{N_f}$ in a triangle see Figure~\ref{fig:ConstructRC} and for $\X(tf)$ see Figure~\ref{fig:TriTilingOneCone}.

Let $t\in \Zgeq$ be big enough, such that by Theorem \ref{Thm:TilingIntro} we have a tiling of  $\upper \mP$, which covers the domain complex, into regions:
\begin{align*}
\{x+R(N_f) : \; f\leq \mP, \, x\in \X(tf)\}.
\end{align*}
To compute $ \DC {t\mP}  = \vol((t\mP \cap \Lambda ) + T)$ we can thus compute the volume in each region and add everything up:
\begin{align} \label{Eqn:VolTiling}
\vol( \DC {t\mP })= \sum_{f \leq \mP} \sum_{x\in \X(tf)}     \vol\big(\left(x+R(f)\right) \cap  \DC {t\mP } \big).
\end{align}
Hence, for $f\leq \mP$ it suffices to show that the volume on the right hand side of the equation equals the \vDC of $N_f$, namely that
\begin{align*}
  \vol\big(\left(x+R(f)\right) \cap  \DC {t\mP }\big)  = v_f
\end{align*}
for all $x\in \X(tf)$. Then Equation (\ref{Eqn:VolTiling}) yields
\begin{align*}
\vol((t\mP \cap \Lambda ) + T)&= \sum_{f\leq \mP} \sum_{x\in \X(tf)}  v_f\\
		&=  \sum_{f \leq \mP} \left| \X(tf) \right|  \cdot v_f 
\end{align*}
as desired. 

We recall the definition of the \vDC $v_f$ as
\begin{equation*}
v_f=  \DC {N_f^\vee }= \vol \big(R(f) \cap ((N_f^\vee \cap \Lambda ) + T)\big).
\end{equation*}
Let $x\in \X(tf)$. Since $(t\mP \cap \Lambda) \subseteq (x+N_f^\vee) \cap \Lambda$, we have
\begin{equation} \label{Eqn:LatPNvee}
(x+R(f))\cap  \DC {t\mP } \subseteq  \big(x+R(f)\big) \cap  \DC {x+N_f^\vee }
\end{equation}
We first want to show that we have equality in (\ref{Eqn:LatPNvee}). To do so, we show that 
\begin{equation} \label{Eqn:RcapT}
(x+R(f)) \cap (y +T)=\varnothing
\end{equation}
for all $y\in ((x+N_f^\vee ) \cap \Lambda)\backslash (t\mP \cap \Lambda)$. 
$y\notin (t\mP \cap \Lambda)$ means there is a vertex $v$ of $\mP$ with $y\notin (N_v^\vee + tv)$ .
By considering $t$  large enough, we can ensure that Equation~(\ref{Eqn:RcapT}) holds for all $y$ with $\displaystyle y\in \bigcap_{v\in f} (N_v^\vee +tv)$  and $y\notin (N_v^\vee +tv)$ for a vertex $v \in \mP$ that is not a vertex of $f$. So let $y\in  ((x+N_f^\vee ) \cap\Lambda)$ with $y\notin  (N_v^\vee + tv)$ for a vertex $v$ of $f$. Then there is a facet $F$ of $\mP$ with $v\in F$ and $y \notin (N_F^\vee +tv)$.
But then $(y-tv)+T \nsubseteq  \Int(N_F^\vee)$ and hence, $(y-tv+T)\subseteq L(F)+R(F)$ (cf. Equation~(\ref{Eqn:L+R}) in the proof of Lemma~\ref{Lm:TilingOneCone}).
Since $x\in \X(tf)$, we have $(x-tv)\in X^{N_v}_{N_f}$ and thus $(x-tv+R(f))\cap (y-tv+T)\neq \varnothing$, which is equivalent to (\ref{Eqn:RcapT}). Hence, equality in Equation (\ref{Eqn:LatPNvee}) follows.

Since $x\in \X(tf) \subseteq \Lambda$, we have $(x+N_f^\vee ) \cap \Lambda=x+(N_f^\vee \cap \Lambda)$ and hence,
\begin{align*} 
\vol\big((x+R(f))\cap \DC {t\mP }\big) & =  \vol\big((x+R(f)) \cap  \DC {x+N_f^\vee }\big)\\
		&=\vol\big(R(f) \cap   \DC {N_f^\vee } \big)\\
		&= v_f,
\end{align*}
as we wanted to show.

\end{proof}

Lemma \ref{Lm:FundInReg} gives a general property of the regions and independent of a concrete polytope, so we use the general notation of cones. 

\begin{lemma} \label{Lm:FundInReg}
For each pointed cone $C$ the fundamental domain  $T(C)$ in the linear space $C^\perp$ is contained in the region $R(C)$.
\end{lemma}

\begin{proof}
We show the statement inductively. Let $C_0=\{0\}$ be the 0-dimensional cone. Then by construction $R(C_0)=T(C_0)$ and the assertion holds. Now, let $C$ be a 1-dimensional pointed cone. Then the corresponding region $R(C)$ is given by 
\begin{displaymath}
R(C)=\big(V\backslash \left(   X^C_{C_0}+ R(C_0)\right)\big)\cap \left(T(C)+\lin(C) \right)\cap \upper {C^\vee}.
\end{displaymath}
Since  $T(C) \subseteq \left(T(C)+\lin(C) \right)$ and also $T(C)\subseteq \upper {C^\vee}$, we only need to show that $T(C)\subseteq V\backslash \left( X^C_{C_0}+ R(C_0)\right)$. 
Therefore, let $p\in T(C)$. Then there is exactly one $x\in \Lambda$ with $p\in (x+R(C_0))$.  
 Since $p\in \bd(C^\vee)$, we have $(x+R(C_0)) \cap \bd(C^\vee)\neq \varnothing$ and thus $x\notin  X^C_{C_0}$.
Hence, \[p\in (x+T(C_0))\subseteq  \big(V\backslash \left( X^C_{C_0}+ R(C_0)\right)\big).\]

Now, let $C$ be any pointed cone with $\dim(C)>1$. We now have
\begin{displaymath}
R(C)=\left(V\backslash \bigcup_{K<C} \left(X_K^C+R(K)\right)\right) \cap \left( T(C) + \lin(C)\right) \cap \upper {C^\vee}.
\end{displaymath}
Again, $T(C)\subseteq \left( T(C) + \lin(C)\right)$ and $T(C)\subseteq \upper {C^\vee}$. Let $p\in T(C)$. 
Assume we have a face $K<C$ and a lattice point $x\in L(K)$ with $p\in x + R(K)$. 
We again want to show that then $x\notin X^C_K$. Since $\dim(C)>1$ we find  a ray $M<C$ incomparable with $K$. The inclusion $T(C)\subseteq C^\perp \subseteq M^\perp$ yields $T(C)\subseteq L(M)+T(M)$. By induction we can assume $T(M)\subseteq R(M)$ and hence, $T(C) \subseteq L(M)+R(M)$.  So for $p\in T(C)$, if $p\in (x+R(K))$ for some $x\in L(K)$, then there exists $y\in L(M)$, such that $p\in (x+R(K)) \cap (y+R(M))$ and hence, $x\notin X^C_K$ by Property~(\ref{Ppt:Nonintersect}) in the construction of $X^C_K$. Hence, we have shown, that $\displaystyle p\in \left(V\backslash \bigcup_{K<C} \left(X_K^C+R(K)\right)\right)$ for any $p\in T(C)$ and thus $T(C)\subseteq R(C)$.
\end{proof}

\begin{lemma} \label{Lm:CoveringOfFace}
There exists a $t_0\in \mathbb{Z}_{>0}$ such that for each $t\geq t_0$ 
the dilation of a face $f<\mP$ by $t$ satisfies 
\begin{equation*}
tf\subseteq \bigcup_{g\leq f}\left( \X(tg)+R(g) \right).
\end{equation*}
\end{lemma}

\begin{proof} 
For $t_0$ big enough, Theorem \ref{Thm:TilingIntro} yields that 
\begin{equation} \label{FullTiling}
tf\subseteq \bigcup_{g  \leq \mP} \left( \X(tg) +R(g) \right)
\end{equation}
for all $t\geq t_0$.
We need to show that in (\ref{FullTiling}) the translated regions of faces $g $ of $ \mP$ with $g\nleq f$ do not intersect with $tf$. We can divide these faces into four groups: the faces $g$ with $f\cap g =\varnothing$; the faces $g$ with $f\cap g \neq \varnothing$ and $f, g$ are incomparable; the face $g=\mP$; and the faces $g$ with $f<g<P$. 

By choosing $t_0$ big enough (cf. proof of Theorem \ref{Thm:TilingIntro}), we can ensure that $(x+R(g))\cap tf = \varnothing$ for all $g\leq \mP$ that do not intersect $f$ and all $\displaystyle x\in \X(tg)$. 

For the second case, let $g\leq \mP$ with $f, g$ incomparable and there exists a vertex $v$ of $\mP$ with $v\in f\cap g$. Let $x\in \X(tg)$. Then $(x-tv)\in X^v_g$. 
In particular (by property (\ref{Ppt:Inside})), that means  $((x-tv)+R(g))\subseteq \Int(N_F^\vee)$ for all facets $F$ that contain $v$ but not $g$.
  But since $tf-tv$ is on the boundary of $N_F^\vee$ and  for at least one facet $F$ containing $v$ that does not contain $g$, we get that $(x-tv+R(g)) \cap (tf-tv) = \varnothing$ and hence, $(x+R(g)) \cap tf = \varnothing$.

For the case $g=\mP$ we note that $R(g)=T$ and then $(x+R(g)) \cap tf = \varnothing$ follows by exactly the same arguments as in the second case. 

In the fourth case, we consider $g$ with $f< g< \mP$. 
Then there exists a facet $F$ of $\mP$ with $f\subseteq g\cap F$ and $g, F$ incomparable. By Lemma \ref{Lm:FundInReg} we have that $T(F)\subseteq R(F)$. Let $v$ be a vertex of $\mP$ with $v\in f$ and let $x\in \X(tg)$. Then  $(x-tv)\in X^v_g$ and thus, by property (\ref{Ppt:Nonintersect})  we have $(y+R(F))\cap (x-tv+R(g))=\varnothing$ for all $y\in L(F)$ and in particular 
\begin{equation} \label{Eqn:foo}
(y+T(N_F))\cap (x-tv+R(g))=\varnothing
\end{equation} 
for all $y\in L(N_F)$. 
Since $f\subseteq F$, we have $tf\subseteq tv+L(F)+T(F)$ which, together with (\ref{Eqn:foo}) shows that  $(x+R(g)) \cap tf = \varnothing$.

Hence, we have shown that for all faces $g$ with $g\nleq f$ the intersection $tf \cap (\X(tg)+R(g))$ is empty and hence, $\displaystyle tf\subseteq \bigcup_{g\leq f}\left( \X(tg)+R(g) \right)$ as we wanted to show.
\end{proof}

\begin{lemma} \label{Lm:Vol(tf)}
There exists a $t_0\in \mathbb{Z}_{>0}$ such that for each $t\geq t_0$ and  every face $f<\mP$
\begin{equation*}
\vol(tf)=\sum_{g\leq f}  \left| \X(tg) \right|\cdot w^g_f .
\end{equation*}
\end{lemma}
In other words, the volume of $tf$ is given by the number of feasible lattice points in $tf$, $\X(tf)$, plus the  \ws for each face $g<f$.

\begin{proof}
We recall that $w^g_f$ is defined in \eqref{Eqn:DefCorrectionVolume} by 
$$
w^g_f := \vol\left(	R(g) \cap N_f^\perp\cap  N_g^\vee \right)
.
$$
From Lemma \ref{Lm:CoveringOfFace} we  deduce
\begin{align*}
\vol(tf)  
	&= \vol\left(	\bigcup_{g\leq f} \left(	(\X(tg)+R(g))\cap tf 	\right)	\right) \\
	&=\sum_{g\leq f} \sum_{x\in \X(tg)} \vol\left(	(x+R(g))\cap tf	\right).
\end{align*}
For  $g\leq f$ we have
\begin{align*}
 \vol\left(	(x+R(g))\cap tf	\right) =  \vol\left(	R(g)\cap (-x+tf)	\right)
\end{align*}
and since $(-x+tf)\subseteq (N_f^\perp\cap N_g^\vee)$ and $ [( N_f^\perp\cap N_g^\vee)\backslash (-x+tf)] \cap R(g) = \varnothing$, we get 
\begin{align*}
 \vol\left(	R(g)\cap (-x+tf)	\right) = w^g_f
\end{align*}
and hence
\begin{align*}
\vol(tf) &= \sum_{g\leq f} \sum_{x\in \X(tg)} \vol\left(	(x+R(g))\cap tf	\right)\\
	& = \sum_{g\leq f}  \left| \X(tg) \right|\cdot w^g_f 
\end{align*}
as we wanted to show.
\end{proof}

\section*{Acknowledgements}

We like to thank the two anonymous referees, as well as Frieder Ladisch and Erik Friese for valuable comments.
Moreover, Maren H.~Ring is grateful for support by a PhD scholarship of Studienstiftung des
Deutschen Volkes (German Academic Foundation).
Both authors gratefully acknowledge support by DFG grant SCHU~1503/6-1.

%%%%%%%%%%%%%%%%%%%%%%%%%%%%%%%%%%%%%%%%%%%%%%%%%%%%%%%%%%%%%%%%%%%%%%%%%%%%%%%%%
%%%%%%%%%%%%%%%%%%%%%%%%%%%%%%%%%%%%%%%%%%%%%%%%%%%%%%%%%%%%%%%%%%%%%%%%%%%%%%%%%

\medskip

\end{document}